\documentclass{amsart}
\usepackage{stmaryrd}
\usepackage[all]{xy}
\usepackage{mathrsfs}
\usepackage{amscd,amssymb,color}

\newcommand{\Sdot}[1][\ssdot]{S_{#1}}

\mathchardef\varDelta="7101

\newcommand{\C}{\mathbf{C}}

\let\sma\wedge
\newcommand{\htp}{\simeq}
\renewcommand{\to}{\mathchoice{\longrightarrow}{\rightarrow}{\rightarrow}{\rightarrow}}

\newcommand{\cA}{{\mathcal A}}
\newcommand{\cB}{{\mathcal B}}
\newcommand{\cC}{{\mathcal C}}
\newcommand{\cD}{{\mathcal D}}

\newcommand{\cM}{{\mathcal M}}
\newcommand{\cS}{{\mathcal S}}
\newcommand{\cT}{{\mathcal T}}
\newcommand{\cU}{{\mathcal U}}

\let\catsymbfont\mathcal
\newcommand{\aA}{{\catsymbfont{A}}}
\newcommand{\aB}{{\catsymbfont{B}}}
\newcommand{\aC}{{\catsymbfont{C}}}
\newcommand{\aD}{{\catsymbfont{D}}}
\newcommand{\aE}{{\catsymbfont{E}}}
\newcommand{\aF}{{\catsymbfont{F}}}
\newcommand{\aG}{{\catsymbfont{G}}}
\newcommand{\aI}{{\catsymbfont{I}}}

\newcommand{\aM}{{\catsymbfont{M}}}

\newcommand{\aO}{{\catsymbfont{O}}}
\newcommand{\aP}{{\catsymbfont{P}}}

\newcommand{\aS}{{\catsymbfont{S}}}
\newcommand{\aT}{{\catsymbfont{T}}}

\newcommand{\aW}{{\catsymbfont{W}}}

\newcommand{\A}{\mathcal{A}}
\newcommand{\B}{\mathcal{B}}
\renewcommand{\C}{\catsymbfont{C}}

\renewcommand{\L}{\mathrm{L}}

\renewcommand{\O}{\mathcal{O}}

\newcommand{\bbK}{I\mspace{-6.mu}K}

\newcommand{\bE}{\mathbb{E}}

\newcommand{\bN}{{\mathbb{N}}}
\newcommand{\bP}{{\mathbb{P}}}

\newcommand{\bU}{\mathbb{U}}
\newcommand{\bZ}{\mathbb{Z}}

\def\quickop#1{\expandafter\DeclareMathOperator\csname
#1\endcsname{#1}}
\quickop{id}\quickop{Id}\quickop{colim}\quickop{hocolim}\quickop{op}
\quickop{co}\quickop{Ar}\quickop{sing}\quickop{Hom}\quickop{w}\quickop{Ho}
\quickop{ob}\quickop{diag}\quickop{Stab}\quickop{Cat}\quickop{Mot}\quickop{Mod}
\quickop{map}\quickop{Cone}\quickop{End}\quickop{Idem}\quickop{Perf}\quickop{Ind}
\quickop{Gap}\quickop{Shv}\quickop{Spaces}\quickop{cof}\quickop{ex}\quickop{perf}
\quickop{tri}\quickop{Set}\quickop{Fib}\quickop{Nat}\quickop{haut}\quickop{holim}\quickop{Tor}\quickop{Ext}\quickop{Diff}\quickop{lax}\quickop{Mult}\quickop{Ob}\quickop{cpt}\quickop{CAlg}\quickop{cf}\quickop{f}\quickop{Fin}\quickop{rex}

\DeclareMathOperator{\cc}{\mathrm{c}}

\newcommand{\sPre}{\mathrm{Pre}_\mathcal{\cS}}
\newcommand{\sflat}{\mathrm{flat}}

\newcommand{\N}{\mathrm{N}}

\newcommand{\Pre}{\mathrm{Pre}}

\newcommand{\icat}{\Cat_\i}

\newcommand{\istabcat}{\Cat_\i^{\ex}}

\newcommand{\idemstabcat}{\Cat_\i^{\perf}}
\newcommand{\idemtimes}{\otimes^{\vee}}

\newcommand{\prcat}{\mathrm{Pr}^\mathrm{L}}
\newcommand{\stabprcat}{\mathrm{Pr}^\mathrm{L}_\mathrm{st}}
\newcommand{\stabcgprcat}{\mathrm{Pr}^\mathrm{L}_{\mathrm{st},\omega}}

\newcommand{\symstabprcat}{\mathrm{CAlg}(\stabprcat)}

\newcommand{\add}{\mathrm{add}}
\newcommand{\loc}{\mathrm{loc}}
\newcommand{\Loc}{\mathrm{Loc}}
\newcommand{\Fun}{\mathrm{Fun}} 
\newcommand{\Map}{\mathrm{Map}}
\newcommand{\bbS}{\mathbb{S}}

\newcommand{\ispec}{\aS_\infty}
\newcommand{\spec}{\aS}
\newcommand{\Spcat}{\Cat_\cS}
\newcommand{\SpcatM}{\Cat_{\aM}}
\newcommand{\Scat}{\Cat_\Delta}

\newcommand{\Alg}{\mathrm{Alg}}

\newcommand{\ie}{{i.e.,}\ }

\newcommand{\too}{\longrightarrow}

\renewcommand{\i}{\infty}

\newcommand{\Motadd}{\cM_{\mathrm{add}}}

\newcommand{\Motloc}{\cM_\mathrm{loc}}

\newcommand{\Umot}{\cU_{\mathrm{add}}}

\newcommand{\Uloc}{\cU_{\mathrm{loc}}}

\numberwithin{equation}{section}

\newtheorem{theorem}[equation]{Theorem}
\newtheorem*{theorem*}{Theorem}

\newtheorem{corollary}[equation]{Corollary}
\newtheorem{lemma}[equation]{Lemma}
\newtheorem{proposition}[equation]{Proposition}
\theoremstyle{definition}
\newtheorem{definition}[equation]{Definition}

\newtheorem{remark}[equation]{Remark}

\newtheorem{notation}[equation]{Notation}

\xyoption{arrow}
\xyoption{matrix}
\xyoption{cmtip}
\SelectTips{cm}{}

\newdir{ >}{{}*!/-5pt/\dir{>}}

\begin{document}

\title[Uniqueness of the multiplicative cyclotomic trace]{Uniqueness of the multiplicative\\ cyclotomic trace}

\author{Andrew J. Blumberg}
\address{Department of Mathematics, University of Texas,
Austin, TX \ 78712}
\email{blumberg@math.utexas.edu}
\thanks{A.~J.~Blumberg was supported in part by NSF grant DMS-0906105 and NSF grant DMS-1151577.}
\author{David Gepner}
\address{Department of Mathematics, Purdue University, West Lafayette
IN 47907}
\email{dgepner@math.purdue.edu}
\thanks{}
\author{Gon{\c c}alo~Tabuada}
\address{Department of
  Mathematics, MIT, Cambridge, MA 02139, United States and
  Departamento de Matematica e CMA, FCT-UNL, Quinta da Torre,
  2829-516, Caparica, Portugal.}
\email{tabuada@math.mit.edu}
\thanks{G.~Tabuada was partially supported by the NEC Award-2742738
  and by the Portuguese Foundation for Science and Technology through
  the project PEst-OE/MAT/UI0297 (CMA).}
\subjclass[2000]{19D10, 19D25, 19D55, 18D20, 55N15}

\keywords{Higher algebraic $K$-theory, stable $\i$-categories,
spectral categories, topological cyclic homology, cyclotomic trace
map}

\begin{abstract}
Making use of the theory of noncommutative motives, we characterize
the topological Dennis trace map as the unique {\em multiplicative}
natural transformation from algebraic $K$-theory to topological
Hochschild homology ($THH$) and the cyclotomic trace map as the unique
multiplicative lift through topological cyclic homology ($TC$).
Moreover, we prove that the space of all multiplicative structures on
algebraic $K$-theory is contractible.

We also show that the algebraic $K$-theory functor from small stable
$\i$-categories to spectra is lax symmetric monoidal, which in
particular implies that $E_n$ ring spectra give rise to $E_{n-1}$ ring
algebraic $K$-theory spectra.  Along the way, we develop a
``multiplicative Morita theory'', establishing a symmetric monoidal
equivalence between the $\i$-category of small idempotent-complete
stable $\i$-categories and the Morita localization of the
$\i$-category of spectral categories.
\end{abstract}

\maketitle
\setcounter{tocdepth}{1}
\tableofcontents

\section{Introduction}\label{sec:intro}

Algebraic $K$-theory provides rich invariants of rings, schemes, and
manifolds, encoding information reflecting arithmetic, geometry, and
topology.  The algebraic $K$-theory of a ring or scheme captures
information about classical arithmetic invariants (e.g., the Picard
and Brauer groups)~\cite{TT}; for a manifold $M$, the algebraic
$K$-theory of $\Sigma^{\infty}_+ \Omega M$ (Waldhausen's $A$-theory)
is closely related to stable pseudo-isotopy theory and $B\Diff(M)$,
the classifying space of the space of diffeomorphisms of
$M$~\cite{WaldA2}.  These seemingly disparate examples are unified by
the perspective that algebraic $K$-theory is a functor of {\em stable
categories}; e.g., a suitable enhancement of the derived category of a
scheme, or the category of modules over a ring spectrum.  This
viewpoint was initiated by Thomason-Trobaugh~\cite{TT}; see
also~\cite{BGT}.

For a commutative ring (or more generally a scheme) $R$, the derived
category of $R$ possesses the additional structure of a symmetric
monoidal tensor product.  Similarly, the category of modules over an
$E_\infty$ ring spectrum is symmetric monoidal.  In this situation,
the algebraic $K$-theory spectrum inherits the structure of a
$E_\infty$ ring spectrum~\cite{BM5, ElmendorfMandell, MayPerm, Wald}.
These results have been important both as a source of structured ring
spectrum models for geometric spectra (e.g., topological $K$-theory) as
well as a source of input to calculation of algebraic $K$-groups.

The main computational tool for understanding algebraic $K$-theory of
rings is the cyclotomic trace map $K \to TC \to THH$ from algebraic
$K$-theory to topological cyclic and Hochschild
homology \cite{BokstedtHsiangMadsen}; this map can be regarded as a
spectrum-level enhancement of the Dennis trace $K \to HC^{-} \to HH$,
and induces an equivalence between relative $K$-theory and
relative $TC$~\cite{DundasFib, McCarthyFib} in many cases.
Multiplicative structures play an essential role in these
calculations, as both $THH$ and $TC$ are commutative ring spectra when
applied to commutative rings and the cyclotomic trace is compatible
with this multiplicative structure (e.g., see \cite[\S
1]{HesselholtMadsenwitt} and \cite[\S 6]{GeisserHesselholt}). 

However, the existing constructions of the multiplicative cyclotomic
trace in the literature are very complicated and apply only in limited
settings.  Hence, it would be very useful to have a general
characterization and construction of the multiplicative cyclotomic
trace map in terms of a simple universal property.  In this paper,
making use of the theory of noncommutative motives, we resolve this
problem: Roughly speaking, the multiplicative topological Dennis trace
is the unique multiplicative natural transformation from algebraic
$K$-theory to $THH$ (see Theorem~\ref{cor:mainone}), and the
multiplicative cyclotomic trace is the unique lifting to $TC$ (see
Theorem~\ref{cor:maintwo}).

\subsection*{Statement of results}

Let $\idemstabcat$ denote the $\i$-category of small
idempotent-complete stable $\i$-categories and exact functors. Recall
from \cite[\S6]{BGT} that a functor $E \colon \idemstabcat \to \cD$
with values in a stable presentable $\i$-category $\cD$, is called an
{\em additive invariant} if it preserves filtered colimits and sends
split-exact sequences of stable $\i$-categories to cofiber sequences of
spectra.  When $E$ moreover sends all exact sequences of
stable $\i$-categories to cofiber sequences, we say it is a {\em localizing
invariant}.  In \cite{BGT} we produced stable presentable
$\i$-categories $\Motadd$ and $\Motloc$ of noncommutative motives and
functors
\[
\xymatrix{
\Umot \colon \idemstabcat \to \Motadd && \Uloc \colon \idemstabcat \to \Motloc
}
\]
characterized by the following universal properties: given any
stable presentable $\infty$-category $\cD$, there are induced
equivalences
\begin{equation}\label{eq:equivalenceadd}
(\Umot)^* \colon \Fun^{\L}(\Motadd,\aD) \stackrel{\sim}{\too} \Fun_{\mathrm{add}}(\idemstabcat,\aD)
\end{equation}
\begin{equation}\label{eq:equivalenceloc}
(\Uloc)^* \colon \Fun^{\L}(\Motloc,\aD) \stackrel{\sim}{\too} \Fun_{\mathrm{loc}}(\idemstabcat,\aD)\,,
\end{equation}
where the left-hand sides denote the $\i$-categories of
colimit-preserving functors and the right-hand sides the $\i$-categories
of additive and localizing invariants.

As stable $\i$-categories, $\Motadd$ and $\Motloc$ carry a natural
enrichment in spectra; see \cite[\S 4]{BGT}.  In \cite{BGT} we showed
that the connective algebraic $K$-theory spectrum functor $K(-)$ and  
the non-connective algebraic $K$-theory spectrum functor $\bbK(-)$ become 
co-representable in $\Motadd$ and $\Motloc$ respectively.  More
precisely, given any idempotent-complete small stable $\i$-category
$\aA$, there are equivalences of spectra 
\begin{eqnarray}\label{eq:corep}
&\Map(\Umot(\ispec^{\omega}), \Umot(\aA)) \simeq K(\aA) 
 &
\Map(\Uloc(\ispec^{\omega}), \Uloc(\aA)) \simeq \bbK(\aA),
\end{eqnarray}
where $\ispec$ denotes the stable $\i$-category of spectra and
$\ispec^{\omega}$ the (essentially small) stable subcategory of
compact objects in $\spec$ (see Theorem 1.3 in~\cite{BGT}). Our first
main result is the following:

\begin{theorem} (see Theorem~\ref{thm:sym})
\label{thm:isym}
The $\i$-categories $\Motadd$ and $\Motloc$ carry natural symmetric
monoidal structures making the functors $\Umot$ and $\Uloc$ symmetric
monoidal.  The tensor units are $\Umot(\ispec^{\omega})$ and
$\Uloc(\ispec^{\omega})$ respectively. 
\end{theorem}

Here the symmetric monoidal structure is induced from the convolution
symmetric monoidal structure on the $\infty$-category $\Pre((\idemstabcat)^{\omega})$ of
presheaves indexed on the symmetric monoidal $\infty$-category of compact objects in $\idemstabcat$.  Similar considerations in the ``dual''
setting of covariant functors from $\Motadd$ and $\Motloc$ to spectra
give rise to symmetric monoidal structures on the $\i$-categories of
additive and localizing invariants:

\begin{theorem}\label{thm:maintwo} (see Theorem~\ref{thm:bodymaintwo})
The $\i$-categories of additive and localizing invariants,
$\Fun_\add(\idemstabcat,\ispec)$ and $\Fun_\loc(\idemstabcat,\ispec)$, are
symmetric monoidal $\i$-categories.
That is, they are the underlying $\i$-categories of
symmetric monoidal presentable stable $\i$-categories
$\Fun_\add(\idemstabcat,\ispec)^{\otimes}$ and
$\Fun_\loc(\idemstabcat,\ispec)^{\otimes}$.  The tensor units are the
connective and non-connective algebraic $K$-theory functors $K$ and
$\bbK$.
\end{theorem}

Theorem~\ref{thm:maintwo} allows us to study $E_{n}$ algebras and
$E_{n}$-maps in $\Fun_{\add}(\idemstabcat,\ispec)$ and
$\Fun_{\loc}(\idemstabcat, \ispec)$.  Since algebraic $K$-theory is the
tensor unit, a consequence of Theorem~\ref{thm:maintwo} and the fact that 
the space of multiplicative maps out of the unit is contractible is
the following strong uniqueness result: 

\begin{corollary}\label{thm:mainthree} (see
Corollary~\ref{cor:bodymainthree})
There exists a unique $E_\infty$ algebra structure on $K$, viewed as
an object of the symmetric monoidal $\i$-category
$\Fun_\add(\idemstabcat, \ispec)^{\otimes}$.  Furthermore, for any
$0\leq n\leq\i$ and any $E_n$ algebra $F$, the space of $E_n$ algebra
maps from $K$ to $F$ is contractible.  Analogous statements hold for
$\bbK$.
\end{corollary}

By combining Theorem \ref{thm:maintwo} with the recent work of
Glasman~\cite{glasman}, we obtain the following relation between
$E_\infty$ algebras and lax symmetric monoidal functors: 

\begin{corollary}
For any presentable symmetric monoidal $\i$-category $\aD$, there are
equivalences of $\i$-categories 
\[
\Alg_{/ E_{\infty}}(\Fun_{\add}(\idemstabcat, \aD)) \stackrel{\sim}{\too}
\Fun^{\lax}_{\mathrm{add}}(\idemstabcat,\aD)
$$
$$
\Alg_{/ E_{\infty}}(\Fun_{\loc}(\idemstabcat, \aD)) \stackrel{\sim}{\too}
\Fun^{\lax}_{\mathrm{loc}}(\idemstabcat,\aD).
\]
\end{corollary}

This leads to the following sharpening of
equivalences~\eqref{eq:equivalenceadd} and~\eqref{eq:equivalenceloc}.

\begin{theorem}\label{thm:mainone}
For any presentable symmetric monoidal $\i$-category $\aD$, there are
equivalences of $\i$-categories 
\[
(\Umot)^* \colon \Fun^{\L,\lax}(\Motadd,\aD) \stackrel{\sim}{\too}
\Fun^{\lax}_{\mathrm{add}}(\idemstabcat,\aD)
\]
\[
(\Uloc)^* \colon \Fun^{\L,\lax}(\Motloc,\aD) \stackrel{\sim}{\too}
\Fun^{\lax}_{\mathrm{loc}}(\idemstabcat,\aD)\,,
\]
where the left-hand sides denote the $\i$-category of lax symmetric
monoidal colimit-preserving functors and the right-hand sides denote
the $\i$-categories of lax symmetric monoidal additive or localizing
invariants, respectively.
\end{theorem}

In order to produce lax monoidal functors of $\i$-categories
$\idemstabcat \to \ispec$, we develop a multiplicative version of Morita
theory for spectral categories.
In~\cite[\S 4]{BGT} we proved that the $\i$-category $\idemstabcat$ of
small stable idempotent-complete $\infty$-categories admits a model
given by the localization of the category $\Spcat$ of spectral
categories (categories enriched in spectra) with respect to the class
$\mathcal{W}$ of Morita equivalences.  Here recall that a spectral
functor $\aC \to \aD$ is a DK-equivalence if it induces weak equivalences
$\aC(X,Y) \to \aD(FX,FY)$ on all mapping spectra and is homotopically
essentially surjective, and a Morita equivalence if it induces a
DK-equivalence on the induced functor $\Mod(\aC) \to \Mod(\aD)$.
In particular, any stable $\i$-category can be ``rigidified'' to a
(pre-triangulated) spectral category.

For our applications herein, we generalize this rigidification result
to the multiplicative setting.  To maintain homotopical control on the
smash product of spectral categories, we use the notion of a flat
spectral category, i.e., a spectral category such that tensoring with
it preserves Morita equivalences.  Since pointwise-cofibrant spectral
categories are flat, and any flat spectral category is DK-equivalent
to a pointwise-cofibrant spectral category, we will write
$\Cat_{\aS}^{\textrm{flat}}$ for the full subcategory of
pointwise-cofibrant spectral categories.  Following \cite{HTT}, we
write $\N$ for the nerve functor from categories (or, more generally,
simplicial categories) to $\infty$-categories.
  
\begin{theorem}\label{thm:multMor} (see Theorem~\ref{thm:symcomp})
There is an equivalence of symmetric monoidal $\i$-categories between
$(\idemstabcat)^{\otimes}$ and
$(\N(\Spcat^{\sflat})[\aW^{-1}])^{\otimes}$, where $\aW$ denotes the
class of Morita equivalences.
\end{theorem}

Since the $\i$-category associated to $\Spcat^{\sflat}$ is a model for
the $\i$-category of spectral categories and Morita equivalences, we
use $(\N(\Spcat^{\sflat})[\aW^{-1}])^{\otimes}$ as a specific model of
the symmetric monoidal $\i$-category of spectral categories and Morita
equivalences.  Hence, Theorem~\ref{thm:multMor} allows us to describe
multiplicative objects in $\Fun_{\add}(\idemstabcat, \ispec)$ and 
$\Fun_{\loc}(\idemstabcat, \ispec)$ as functors from spectral categories
to spectra.  Specifically, we prove in section~\ref{sec:thhmult} that
suitable homotopical point-set lax symmetric monoidal functors from
spectral categories to spectra give rise to $E_\infty$ algebras in
$\Fun_{\add}(\idemstabcat, \ispec)$ and $\Fun_{\loc}(\idemstabcat, \ispec)$.

\begin{theorem}\label{thm:new} (see Theorem~\ref{thm:laxdescends})
Let $E$ be a lax symmetric monoidal functor from spectral categories
to spectra.  Further assume that $E$ preserves Morita equivalences between
flat spectral categories and that the induced functor
$\tilde{E} \colon \idemstabcat \to \ispec$ is an additive invariant.
Then $\tilde{E}$ naturally extends to an $E_\i$ algebra object of
$\Fun_\add(\idemstabcat, \ispec)^{\otimes}$.  The analogous results
for localizing invariants hold.
\end{theorem}

Our main example of a functor $E$ satisfying the hypotheses of
Theorem~\ref{thm:new} is topological Hochschild homology; see
corollary~\ref{cor:thheinf}.

\subsection*{Applications}

Recall that one of the most interesting applications of the theory
developed in \cite{BGT} was the proof that the set of homotopy classes
of natural transformations from $K$ to $THH$ is isomorphic to $\bZ$,
with the topological Dennis trace corresponding to $1 \in \bZ$.  The
following consequence of Corollary~\ref{thm:mainthree} and
Theorem~\ref{thm:new}, which is our main application, is a significant
sharpening of this result.  

\begin{theorem}\label{cor:mainone} (see Theorem~\ref{thm:cormainone}) 
The space of maps of $E_\infty$ algebras from $K$ to $THH$ in
$\Fun_\add(\idemstabcat, \ispec)^{\otimes}$ is contractible.
Equivalently, the space of natural transformations of lax symmetric
monoidal functors from $K \to THH$ in
$\Fun^{\lax}_{\mathrm{add}}(\idemstabcat,\aS_\infty)$ is contractible.  The
unique element of this space is the topological Dennis trace map. The
analogous result holds for $\bbK$.
\end{theorem}

Although $TC$ is not an additive invariant (as it does not preserve
filtered colimits), we can extend this result to an identification of
the cyclotomic trace.  Specifically, $TC$ can be described as
$\holim_n TC^n$, where the $TC^n$ are additive invariants, and this
characterization gives rise to the following result.

\begin{theorem}\label{cor:mainoneprime} (see Theorem~\ref{thm:cormaintwo})
The space of maps of $E_\infty$ algebras from $K \to TC^n$ in
$\Fun_\add(\idemstabcat, \ispec)^{\otimes}$ is contractible.
Equivalently, the space of natural transformations of lax symmetric
monoidal monoidal functors from $K \to TC^n$ in
$\Fun^{\lax}_{\mathrm{add}}(\idemstabcat,\aS_\infty)$ is contractible.  The
unique homotopy class of maps of $E_\infty$ algebras in
$\Fun(\idemstabcat, \ispec)^{\otimes}$ from $K$ to $TC$ that restrict
to maps of $E_\infty$ algebras $K \to TC^n$ is the cyclotomic trace.
\end{theorem}

Our second application, which generalizes results
of~\cite{BM5,ElmendorfMandell}, is the following: 

\begin{theorem}\label{thm:ikmon} (see Proposition~\ref{prop:kmonoidal}
and Corollary~\ref{cor:kmon})
The algebraic $K$-theory functors are lax symmetric monoidal as
functors from $\idemstabcat$ to $\ispec$.  In particular, for every
$E_n$ object $\aA$ in $\idemstabcat$, with $0\leq n\leq\infty$,
$K(\aA)$ and $\bbK(\aA)$ are $E_n$ ring spectra.
\end{theorem}

Theorem~\ref{thm:ikmon}, combined with Lurie's
proof~\cite[8.1.2.6]{HA} of a conjecture of
Mandell~\cite[5.3]{MandellE1234}, which asserts that an
$E_{n+1}$ algebra in spectra gives rise to an $E_{n}$ category of
compact modules in $\idemstabcat$, implies that the $K$-theory of an
$E_{n+1}$ ring spectrum is an $E_{n}$ ring spectrum.  Our third
application, which is a consequence of theorem~\ref{thm:isym} and
equivalences \eqref{eq:corep}, is the following:

\begin{corollary} (see Corollary~\ref{cor:Aenriched})
\label{cor:iAenriched}
The symmetric monoidal homotopy categories $\mathrm{Ho}(\Motadd)$ and
$\mathrm{Ho}(\Motloc)$ are enriched over the symmetric monoidal
homotopy category $\mathrm{Ho}(\Mod_{A(*)})$ of $A(*)$-modules, where
$A(\ast)=K(\bbS)\simeq \bbK(\bbS)$.
\end{corollary}

Our final application is a consistency result for multiplicative
structures on algebraic $K$-theory.  Associated to a spectral category
$\aC$ is its pre-triangulated spectral category $\Perf(\aC)$ of
(homotopically) compact $\aC$-modules, which has a Waldhausen
structure inherited from the projective model structure on
$\aC$-modules. There are now two possible constructions of algebraic
$K$-theory landing in the $\i$-category of spectra. This is depicted
as follows:
\[
\xymatrix{
\N(\Spcat) \ar[r]^{\Perf(-)} & \N(\Spcat) \ar[rrrrr]^{K(-)} \ar[d] &&&&& \ispec \ar@{=}[d]
\\
& \idemstabcat \ar[rr]^{\Umot(-)} && \Motadd
\ar[rrr]^{\Map(\Umot(\ispec^{\omega},-))} &&& \ispec. \\
}
\]
In \cite[\S 7]{BGT} it was proved that these two approaches (as
well as a third $\i$-categorical version of the Waldhausen
construction) are canonically equivalent.  The following corollary of
Corollary~\ref{thm:mainthree} promotes this equivalence to the
multiplicative setting. 

\begin{corollary}\label{cor:maintwo}(see
Corollary~\ref{cor:bodymaintwo})
Let $\aC$ be a symmetric monoidal spectral category, and
$\Perf(\aC)$ be the resulting symmetric monoidal category of
compact modules.  The two algebraic $K$-theory spectra described above 
are naturally equivalent as $E_\infty$ algebras in the $\i$-category
of spectra.  If $\aC$ is a monoidal spectral category, the resulting
algebraic $K$-theory spectra are naturally equivalent as
$A_\infty$ algebras in the $\i$-category of spectra. Analogous results
hold for $\bbK$.
\end{corollary}

Finally, we note that there has been recent work on the subject of
multiplicative structures on an $\i$-categorical model of algebraic
$K$-theory due to Barwick~\cite{Barwick}.  Naturally, the basic
results are broadly similar, but certain technical differences arise
in both the method of proof and the approach to describing the input
data for algebraic $K$-theory.

\subsection*{Acknowledgments}
The authors would like to thank Samuel Isaacson for sharing his
insights and in particular for some extremely useful suggestions. They
would also like to thank Vigleik Angeltveit for asking motivating
questions, Clark Barwick, Jacob Lurie, David Ben-Zvi, and Michael
Mandell, for many helpful discussions, and the Midwest Topology
Network for funding various trips which facilitated this research.
The authors would like also to thank an anonymous referee for a
careful reading and detailed comments and corrections which
substantially improved this paper.

\subsection*{Notations}
Throughout the article we will use the letter $\cT$ to denote the
symmetric monoidal simplicial model category of simplicial sets and
the letter $\cS$ for the symmetric monoidal simplicial model category
of symmetric spectra of simplicial sets~\cite{HSS}.  

\section{Background on spectral categories}\label{sec:background}

Our work depends on a careful analysis of the interplay between
different models of the homotopy theory of stable homotopy categories.
In this section, we briefly review the details of the model given by
spectral categories.  Other references on spectral categories
include \cite[Section~2]{BM}, \cite[Section~2.1]{BGT}, \cite[Appendix
A]{SS}, or \cite[Section 2]{Spectral}.

Recall that a {\em spectral category} $\cA$ is a category enriched
in the category $\cS$ of symmetric spectra. Concretely, it consists of the following data: 
\begin{itemize}
\item A class of objects $\mathrm{obj}(\cA)$;
\item A symmetric spectrum $\cA(x,y)$ for each ordered pair of objects $(x,y)$;
\item Composition morphisms in $\cS$
\begin{eqnarray*}
\cA(y,z) \wedge \cA(x,y) \too \cA(x,z) && x, y \in \mathrm{obj}(\cA)
\end{eqnarray*}
satisfying the usual associativity condition;
\item Unit morphisms $\bbS \to \cA(x,x), x \in \mathrm{obj}(\cA)$, satisfying the usual unit condition with respect to the above composition. 
\end{itemize}

A spectral category is called {\em small} whenever its class of
objects forms a (small) set.  A {\em spectral functor} $F\colon\cA \to \cB$
is a functor enriched over $\cS$.  Concretely, it consists of a map
$\mathrm{obj}(\cA) \to \mathrm{obj}(\cB)$ and of morphisms in $\cS$ 

\begin{eqnarray*}
F(x,y)\colon \cA(x,y) \too \cB(Fx,Fy) && x, y \in \mathrm{obj}(\cA)
\end{eqnarray*}

satisfying the usual unit and associativity conditions.

\begin{notation}
Let $\Spcat$ denote the category of small spectral categories.
\end{notation}

Given a small spectral category $\cA$, one can form a genuine category
$[\cA]$ by keeping the same set of objects and by defining
$[\cA](x,y)$ as the set of morphisms in the homotopy category
$\Ho(\cS)$ from the sphere spectrum $\bbS$ to $\cA(x,y)$. This gives
naturally rise to a well-defined functor 
\[
[-]\colon \Spcat \too \Cat
\]
with values in the category of small categories.

We now turn to the homotopy theory of spectral categories.

\begin{definition}
A spectral functor $F\colon \cA \to \cB$ is called a {\em $DK$-equivalence} if:
\begin{itemize}
\item The morphisms in $\cS$
\begin{eqnarray*}
F(x,y)\colon \cA(x,y) \too \cB(Fx,Fy) && x,y \in \mathrm{obj}(\cA)
\end{eqnarray*}
are stable equivalences of spectra;
\item The induced functor $[F] \colon [\cA] \to [\cB]$ is an
equivalence of categories. 
\end{itemize}
\end{definition}

As proved in~\cite[Thm.~5.10]{Spectral}, $\Spcat$ carries a (right
proper) Quillen model structure whose weak equivalences are the
DK-equivalences.  Moreover, the criteria
of~\cite[Thm.~2.5,Thm.~2.20]{BergerMoerdijk} implies that this model
structure is in fact cofibrantly-generated.

For the purposes of algebraic $K$-theory (and related functors), it is
convenient to work with a weaker notion of equivalence than
DK-equivalence.  Given a spectral category $\cA$, a {\em $\cA$-module}
is a spectral functor from the opposite spectral spectral category
$\cA^{\op}$ (where $\cA^{\op}(x,y)\colon =\cA(y,x)$) to the spectral category
$\cS$ of symmetric spectra. Let us denote by $\widehat{\cA}$ the
spectral category of $\cA$-modules. As explained in \cite[A.1.1]{SS},
$\widehat{\cA}$ carries a (combinatorial) spectral model structure in
which the weak equivalences are the pointwise stable equivalences and
the fibrations are the pointwise fibrations. In what follows we will
denote by $\cD(\cA)$ the {\em derived category} of $\cA$, \ie the
homotopy category $\Ho(\widehat{\cA})$ associated to this model
structure. Note that one has a (fully faithful) spectral Yoneda
embedding
\begin{eqnarray*}
\cA \too \widehat{\cA} && z \mapsto \cA(-,z)\,.
\end{eqnarray*}

Let $\cD_{\tri}(\cA)$ denote the smallest triangulated subcategory of
$\cD(\cA)$ containing the image of $\cA$ under the Yoneda embedding,
and $\cD_{\perf}(\cA)$ the smallest thick triangulated subcategory of
$\cD(\cA)$ containing the image of $\cA$ under the Yoneda embedding. 

Note that every spectral functor $F\colon \cA \to \cB$ gives rise to a 
restriction/extension Quillen adjunction  
\[
\xymatrix{
\widehat{\cB}\ar@<1ex>[d]^-{F^\ast} \\
\widehat{\cA}\ar@<1ex>[u]^-{F_!}\,.
}
\]
We hence obtain a total left-derived functor $\mathbb{L}
F_!\colon \cD(\cA) \to \cD(\cB)$ which restricts to $\mathbb{L}
F_!\colon \cD_{\perf}(\cA) \to \cD_{\perf}(\cB)$ and furthermore to
$\mathbb{L} F_!\colon \cD_{\tri}(\cA) \to \cD_{\tri}(\cB)$.  
\begin{definition}
A spectral functor $F\colon \cA \to \cB$ is called:
\begin{itemize}
\item A {\em triangulated equivalence} if $\mathbb{L}
F_!\colon \cD_{\tri}(\cA) \stackrel{\sim}{\to} \cD_{\tri}(\cB)$ is an
equivalence; 
\item A {\em Morita equivalence} if $\mathbb{L}
F_!\colon \cD_{\perf}(\cA) \stackrel{\sim}{\to} \cD_{\perf}(\cB)$ is
an equivalence. 
\end{itemize} 
\end{definition}

As we shall recall below, with these notions of equivalence spectral
categories provide a point-set model of the homotopy theory of small
stable homotopical categories.  To this end, it is sometimes
convenient to work with a variant model of spectral categories.  Let
$\aM$ denote the symmetric monoidal category of
$S$-modules of Elmendorf-Kriz-Mandell-May~\cite{EKMM} and $\SpcatM$ denote the category of small
categories enriched in $\aM$ and the spectral (enriched) functors.
The notions of DK-equivalence, triangulated equivalence, and Morita
equivalence generalize in the obvious fashion.  Recall
from~\cite[3.7]{MM} that there is a Quillen equivalence $(\bN \circ
\bP, \bU \circ \bN^{\sharp})$ connecting $\cS$ and $\aM$.  The left
adjoint is strong monoidal and the right adjoint is lax
monoidal. 

\begin{proposition}
The induced adjoint pair of functors between $\Spcat$ and
$\SpcatM$ produces a transferred model structure on $\SpcatM$ in
which the weak equivalences are the DK-equivalences.  This model
structure is Quillen equivalent to the model structure on $\Spcat$.
\end{proposition}

\begin{proof}
We can apply standard transfer arguments to lift the DK-equivalence
model structure on $\Spcat$ to a model structure on $\SpcatM$
(e.g., see~\cite[\S 2.5]{BergerMoerdijkaxiom}).  We define weak
equivalences and fibrations in $\SpcatM$ to be maps which are weak
equivalences and fibrations via $\bU \circ \bN^{\sharp}$.  Since the
right adjoint $\bU \circ \bN^{\sharp}$ creates the weak equivalences,
preserves sequential colimits along closed inclusions, all objects in
$\SpcatM$ are fibrant, and there are functorial path objects (via the
construction of~\cite[4.1.1]{tabmatrix}), this specifies a
transferred model structure on $\SpcatM$.  Since $\aM$ and $\cS$
are Quillen equivalent, it is clear that the transferred model
structure on $\SpcatM$ is Quillen equivalent to $\Spcat$.
\end{proof}

The category $\SpcatM$ allows us to correct a small error in
~\cite{BGT}.  Specifically, we studied therein the construction
$\psi_{\perf}$ that takes a small pointwise-cofibrant spectral
category $\aC$ to the cofibrant, fibrant, homotopically compact
objects in the projective model structure on $\Mod(\aC)$~\cite[\S
  4.1]{BGT}.  A functor $F \colon \aC \to \aD$ induces a functor
$\Mod(\aC) \to \Mod(\aD)$ via left Kan extension followed by fibrant
replacement.  Because of the appearance of fibrant replacement,
$\psi_{\perf}$ is not a strict endofunctor on $\Spcat$, although it
does have suitable coherence to give rise to an $\i$-functor.
However, since all objects in $\SpcatM$ and in the variant of
$\Mod(\aC)$ in this setting are fibrant, the fact that the left Kan
extension is a Quillen left adjoint (and so preserves cofibrant
objects and hence cofibrant-fibrant objects) implies that
$\psi_{\perf}$ is functorial in this setting.  An analogous issue
arises with the functoriality of the spectral
envelope~\cite[4.10]{BGT}, and can be corrected in the same fashion.

\section{Background on $\i$-categories and $\i$-operads}

The basic setting for our work is the theory of $\i$-categories (and
particularly stable $\i$-categories), which provide a tractable way to
handle the abstract homotopy theory of the ``category of homotopical
categories'' as well as categories of homotopical functors.  There are
now many competing models of $\i$-categories, including Rezk's Segal
spaces~\cite{Rezk}, the Segal categories~\cite{HirschowitzSimpson,
  Tamsamani} of Simpson and Tamsamani, the ``quasicategories'' (weak
Kan complexes) introduced by Boardman and Vogt and studied by Joyal
and Lurie~\cite{BoardmanVogt, Joyal, HTT}, the homotopy theory of
simplicial categories as studied by Dwyer-Kan and Bergner
~\cite{DwyerKan, Bergner}, and others, all of which are known to be
equivalent (see \cite{BergnerCompare} for a nice discussion of the
situation).  We have chosen to work in this paper with the theories of
quasicategories and spectral categories.
Our basic references for the former material are Lurie's
books \cite{HTT, HA}.  In this section we give a brief review of
certain essential foundational aspects of the theory of
quasicategories and then review the theory of $\i$-operads as we 
will apply it in the body of the paper, following \cite[\S 2]{HA}.

We begin by recalling the passage from categories with weak
equivalences (e.g., model categories) to $\i$-categories in the
setting of quasicategories.  One way to produce an $\i$-category from
a category $\aC$ with weak equivalences $w\aC$ is to take a fibrant
replacement of the Dwyer-Kan simplicial localization of $(\aC,w\aC)$
and apply the coherent nerve functor $\N$.  In general, for a category
$\aC$ with weak equivalences we will denote this process by
$\N(L^H\aC)$.  For the purposes of studying the multiplicative
structure induced by a monoidal product, it is convenient to use the
repackaging of this approach to passing from a model category to an
$\i$-category described in \cite[\S 1.3.4 and \S 4.1.3]{HA}.  The
construction of \cite[Construction~4.1.3.1]{HA} produces from an
$\i$-category $\aC$ and a suitable collection of weak equivalences
$\aW$ an $\i$-category $\aC[\aW^{-1}]$.  In particular, given a model
category $\aD$ which is not necessarily simplicial, the coherent nerve
of the subcategory of cofibrant objects $\aD^{\cc}$ is an
$\i$-category and the $\i$-category $\N(\aD^{\cc})[\aW^{-1}]$ is a
version of the underlying $\i$-category of $\aD$.

\begin{proposition}\label{prop:consist}
Let $\aC$ be a combinatorial model category with weak equivalences
$\aW$.  Then there is a categorical equivalence
$\N(\aC^{\cc})[\aW^{-1}] \to \N(L^H\aC)$ induced by the inclusion
$\aC^{\cc} \to L^H \aC^{\cc}$.
\end{proposition}

\begin{proof}
This follows from the fact that the map $\N(\aC) \to \N(L^H \aC)$
exhibits $\N(L^H \aC)$ as the $\i$-category obtained from $\N(\aC)$ by
inverting the morphisms of $\aW$, in the sense of \cite[Definition~1.3.4.1]{HA}.
To see this, recall from \cite{dugger} that the combinatorial model category $\aC$ is
Quillen equivalent to a combinatorial simplicial model category
$\aC'$.  By \cite[Lemma~1.3.4.21]{HA}, this induces an
equivalence of $\i$-categories
\[
\N(\aC^{\cc})[\aW^{-1}] \to \N((\aC')^{\cc})[\aW^{-1}].
\] 
Next, \cite[Theorem~1.3.4.20]{HA} implies that there is an equivalence of
$\i$-categories 
\[
\N((\aC')^{\cc})[\aW^{-1}] \to \N((\aC')^{\circ}),
\]
where $(\aC')^{\circ}$ denotes the full simplicial subcategory consisting
of the cofibrant-fibrant objects in $\aC'$.
Thus, it suffices to compare $\N(L^H \aD)$ and $\N(\aD^{\circ})$ for a
simplicial model category $\aD$.  By \cite[4.8]{DKfunc} there is a
natural zig-zag of equivalences of simplicial categories connecting
$\aD^{\circ}$ and $L^H \aD$, and so the result follows. 
\end{proof}

Next, we briefly review the definitions of stable
$\i$-categories.  An $\i$-category that has finite colimits and limits
is stable~\cite[1.1.1.9]{HA} when pushout and pullback squares
coincide~\cite[1.1.3.4]{HA}; it is straightforward to check that a
stable $\i$-category has a triangulated structure on its homotopy
category coming from the cofiber sequences~\cite[1.1.2.13]{HA}.  A
functor between stable $\i$-categories is exact when it preserves
finite colimits~\cite[\S 1.1.4]{HA}.  An exact functor is an
equivalence when it induces an equivalence on the underlying
(triangulated) homotopy categories).  One of the basic attractive
aspects of the theory of quasicategories is that the quasicategory of
quasicategories is tractable.  Herein, we are particularly interested
in $\istabcat$ and $\idemstabcat$, which are respectively the
$\i$-categories of small stable and idempotent-complete small stable
$\i$-categories (with the equivalences given as above).

We now turn to the theory of $\i$-operads.
Let $\Fin_*$ denote the category with objects the pointed sets
$\langle n\rangle=\{*,1,2,\ldots,n\}$ and morphisms those functions
which preserve the basepoints $*$.  (Classically, this category is
also denoted $\Gamma^{\op}$.)  Recall that $\N$ will denote the nerve
functor; the nerve of an ordinary category is a
quasicategory~\cite[1.1.5.5]{HTT}.  In mild abuse of notation, we will
also use $\N$ to denote the homotopy coherent nerve functor from
simplicial categories to quasicategories~\cite[1.1.2.6]{HTT}.  This is
reasonable since the homotopy coherent nerve of an ordinary category
regarded as a discrete category coincides with the standard
nerve~\cite[1.1.5.8]{HTT}.

Before providing the definition of an $\i$-operad, we recall two basic
definitions.  First, a map $f \colon \langle m \rangle \to \langle
n \rangle$ is inert if $f^{-1}(i)$ has precisely one element for
$i \neq *$~\cite[2.1.1.8]{HTT}.  There are distinguished inert
morphisms $\rho^i\colon \langle n\rangle\to\langle 1\rangle$ that take
everything to the basepoint except $i$, which it taken to $1$.
Second, we need to recall the definition of $p$-coCartesian morphisms.
Given an object $x \in \aC$, let $\aC_{x/}$ denote the category of
objects under $x$.  Similarly, given a morphism $f$ in an
$\i$-category $\aC$, let $\aC_{f/}$ denote the category of morphisms
under $f$.  Then given a functor $p \colon \aC \to \aD$, a morphism
$f \colon x \to y$ in $\aC$ is $p$-coCartesian (lifting $p(f) \colon
px \to py$) if the natural map
\[
\aC_{f/} \to \aC_{x/} \times_{\aD_{px /}} \aD_{p(f) /}
\]
is a trivial fibration of simplicial sets~\cite[\S
2.4.1]{HTT}.  A map $p \colon \aC \to \aD$ is a coCartesian fibration
if it is an inner fibration (has the right lifting property with
respect to inner horn inclusions) and for each $c \in \aC$ and map
$pc \to d$, there is a $p$-coCartesian edge $c \to c'$ that lifts the
given map; see~\cite[\S 2.4]{HTT} for more details.

\begin{definition}{(see \cite[Definition~2.1.1.10]{HA})}\label{defn:iop}
An $\i$-operad is then an $\i$-category $\aO^{\otimes}$ and a functor
$p \colon \aO^{\otimes} \to \N(\Fin_*)$ satisfying the following
conditions: 
\begin{enumerate}
\item[(1)] For every inert morphism $f\colon \langle
m\rangle\to\langle n\rangle$ in $\Fin_*$ and every object $C$ of
$\O^\otimes_{\langle m\rangle}$, there exists a $p$-coCartesian
morphism $\tilde{f} \colon C\to C'$ in $\O^\otimes$ lifting $f$
and hence an induced functor $f_! \colon \O^\otimes_{\langle
m\rangle}\to\O^\otimes_{\langle n\rangle}$. 
\item[(2)] Let $C\in\O^\otimes_{\langle m\rangle}$ and
$C'\in\O^\otimes_{\langle n\rangle}$ be objects, let $f \colon \langle
m\rangle\to\langle n\rangle$ be a morphism in $\Fin_*$, and let
$\map^f_{\O^\otimes}(C,C')$ denote the union of the components of
$\map_{\O^\otimes}(C,C')$ which lie over $f\in\hom_{\Fin_*}(\langle
m\rangle, \langle n\rangle)$. Choose $p$-coCartesian morphisms $C'\to
C'_i$ lying over the inert morphisms $\rho^i\colon \langle
n\rangle\to\langle 1\rangle$ for each $1\leq i\leq n$. Then the
induced map \[\map^f_{\O^\otimes}(C,C')\to\prod_{1\leq i\leq
n}\map^{\rho^i\circ f}_{\O^\otimes}(C,C'_i)\]
is a homotopy equivalence.
\item[(3)] For every finite collection of objects $C_1,\ldots C_n$ of
$\O^\otimes_{\langle 1\rangle}$, there exists an object $C$ of
$\O^\otimes_{\langle n\rangle}$ and $p$-coCartesian morphisms $C\to
C_i$ covering $\rho^i$, $1\leq i\leq n$. 
\end{enumerate}
\end{definition}

This is the generalization of the notion of a multicategory (colored
operad); to obtain the generalization of an operad we restrict to
$\i$-operads equipped with an essentially surjective functor
$\Delta^{0} \to p^{-1}(\langle 1 \rangle)$.  To make sense of this, note that
$p^{-1}(\langle 1 \rangle)$ should be thought of as the ``underlying''
$\i$-category associated to $\aO^{\otimes}$, which should contain only a
single (equivalence class of) object if we're interested in studying the
$\i$-version of an ordinary operad.

More precisely, given a multicategory $\aA$, we can construct a
category $\widetilde{\aA}$ as follows: the objects are finite sets of
objects in $\aA$, and morphisms from $\{X_1, \ldots,
X_m\} \to \{Y_1, \ldots, Y_n\}$ are specified by maps $\langle
m \rangle \to \langle n \rangle$ in $\Fin_*$ and a collection of
morphisms $\{\phi_j \in \aA(\{X_i\}_{i \in \alpha^{-1}(j)},
Y_j)\}_{1 \leq j \leq n}$.  Composition is determined by the
composition laws in the multicategory.  There is a natural functor
$\widetilde{\aA} \to \Fin_*$; Definition~\ref{defn:iop} is modeled on
this structure.  It is straightforward to check that this
construction, the ``category of operations'', is functorial.

More generally, given a simplicial multicategory, there is an analogue
of this construction which yields a simplicial
category~\cite[2.1.1.22]{HA}, and applying the homotopy coherent nerve
to the resulting category yields an $\i$-operad provided that each
morphism simplicial set of the multicategory is a Kan complex;
consult~\cite[Proposition~2.1.1.27]{HA} for further discussion or
see~\cite[\S2]{heuts}.  This construction is sometimes referred to as
the {\em operadic nerve} of the simplicial multicategory.

We now turn to some examples of interest.  The identity map
$\N(\Fin_*) \to \N(\Fin_*)$ is an $\i$-operad; this is the analogue of
the $E_\infty$ operad.  More generally, for each $0\leq n\leq\infty$,
we can define a topological category $\tilde{\bE}[n]$ as follows (see
also~\cite[Definition~5.1.0.2]{HA}).  The objects of $\tilde{\bE}[n]$
are the objects of $\Fin_*$.  Morphisms consist of maps
$\alpha \colon \langle n \rangle \to \langle m \rangle$ along with for
each non-basepoint $j \in \langle m \rangle$ disjoint rectilinear embeddings
$(0,1)^n \to (0,1)^n$ (i.e., maps given by component-wise linear maps)
for each element of $\alpha^{-1}(j)$.  Composition is defined in the
usual fashion.  Using the singular complex functor to get a simplicial
category and applying the homotopy coherent nerve, there is a natural
functor $\N(\tilde{\bE}[n]) \to \N(\Fin_*)$ which is an
$\i$-operad~\cite[5.1.0.3]{HA}.  We refer to the resulting $\i$-operad as
the $E_n$ operad.

Given an $\i$-operad $q\colon \O^\otimes\to\N(\Fin_*)$ and a
coCartesian fibration $p\colon \C^\otimes\to\O^\otimes$ from an
$\i$-category $\C^\otimes$, we say that {\em
$p\colon \C^\otimes\to\O^\otimes$ is an $\O$-monoidal $\i$-category}
if the composite $q\circ p\colon \C^\otimes\to\O^\otimes\to\N(\Fin_*)$
is an $\i$-operad. Such a map $p$ is called a {\em
coCartesian fibration of $\i$-operads}.  In particular, a
{\em symmetric monoidal $\i$-category} is an $\infty$-operad
$\aC^\otimes$ such that the structure map is a coCartesian fibration
of $\i$-operads (see \cite[Example~2.1.2.18]{HA})
\[
p \colon \aC^{\otimes} \to \N(\Fin_*).
\]
The underlying $\i$-category is obtained as the fiber $\aC =
p^{-1}(\langle 1 \rangle)$.  More generally, the fiber over $\langle
n \rangle$ is equivalence to the $n$-fold product of $\aC$.  In abuse
of terminology, we will say that an $\i$-category $\aC$ is a symmetric
monoidal $\i$-category if it is equivalent to $p^{-1}(\langle
1 \rangle)$ for some symmetric monoidal $\i$-category $\aC^{\otimes}$.

Recall from \cite[Example~2.1.1.5]{HA} that given a symmetric monoidal
category $\aC$, there is an associated multicategory in which the
multihomomorphisms are given by the formula
\[
\hom((X_1,\ldots,X_n),Y) =\hom(X_1\otimes\cdots\otimes X_n,Y).
\]
Associated to this multicategory we can construct a category
$\aC^{\otimes}\to\Fin_*$ over
$\Fin_*$ (see \cite[Construction~2.1.1.7]{HA}) such that the
coherent nerve $\N(\aC^{\otimes})\to\N(\Fin_*)$ exhibits
$\N(\aC^{\otimes})$ as a symmetric monoidal $\i$-category;
see \cite[Examples 2.1.1.21 and 2.1.2.22]{HA}.

When $\aC$ is a symmetric monoidal model category, we can also obtain
a symmetric monoidal $\i$-category using the coherent nerve.
Specifically, there is a symmetric monoidal $\i$-category
$(\N(\aC^{\cc})[\aW^{-1}])^{\otimes}$ with underlying $\i$-category
$\N(\aC^{\cc})[\aW^{-1}]$; see \cite[Proposition~4.1.3.4 and Example
4.1.3.6]{HA}.  If $\aC$ is simplicial, then there is an equivalence of
symmetric monoidal $\i$-categories between
$(\N(\aC^{\cc})[\aW^{-1}])^{\otimes}$ and the operadic nerve
$\N^\otimes(\aC^\circ)$ (cf. \cite[Definition 2.1.1.23]{HA}) of
$\aC^{\circ}$, the full subcategory of cofibrant-fibrant objects of
$\aC$ by \cite[Corollary 4.1.3.16]{HA}.

An {\em $\infty$-operad map}
\[
f\colon \O^\otimes\to\O'^\otimes
\]
is a map of simplicial sets $f\colon \O^\otimes\to\O'^\otimes$ over
$\N(\Fin_*)$ such that $f$ takes inert morphisms in $\O^\otimes$ to
inert morphisms in $\O'^\otimes$. The $\i$-category of $\i$-operad
maps is written $\Alg_{\O}(\O')$ and is defined to be the full
subcategory of $\Fun_{\N(\Fin_*)}(\O^\otimes,\O'^\otimes)$ spanned by
the $\infty$-operad maps.  (Here $\Fun_{\N(\Fin_*)}$ denotes maps over
$\N(\Fin_*)$.)

More generally, we can work over a fixed $\infty$-operad $\O^\otimes$.
If $p\colon \C^\otimes\to\O^\otimes$ is an $\infty$-operad map such
that $p$ is also a categorical fibration (a {\em fibration of
$\i$-operads}) and $\alpha\colon \O'^\otimes\to\O^\otimes$ is an arbitrary
$\i$-operad map, then $\Alg_{\O'/\O}(\C)$ is the full subcategory of
$\Fun_{\O^\otimes}(\O'^\otimes,\C^\otimes)$ spanned by the
$\infty$-operad maps; see \cite[Definitions 2.1.2.7 and
2.1.3.1]{HA}.
An object of $\Alg_{\O'/\O}(\C)$ is referred to as an $\O'$-algebra
object of $\C$ over $\O$.  Note that the $\infty$-category
$\Alg_{\O'/\O}(\C)$ is the fiber of the projection
$\Alg_{\O'}(\C)\to\Alg_{\O'}(\O)$ over $\alpha$.  Also, when
$\alpha\colon \O'^\otimes\to\O^\otimes$ is the identity, we write
$\Alg_{/\O}(\C)$ in place of $\Alg_{\O'/\O}(\C)$.

For example, the $\i$-category of commutative algebras in a symmetric
monoidal $\i$-category $\aC^{\otimes}$ is given by suitable sections
of the coCartesian fibration $\aC^{\otimes} \to \N(\Fin_*)$.  The data
of such a section at $\langle n \rangle$ is $n$ copies of the the
underlying object (the section evaluated at $\langle 1 \rangle$), and the
lifting condition for $p$-coCartesian edges provides the
multiplications.  See~\cite[\S 4.2]{groth} for a nice discussion in
detail.

\begin{definition}
Let $\O^\otimes$ be an $\i$-operad with a single object and let $\aC$
be an $\aO$-monoidal $\i$-category.  There is a natural map from  
$\Alg_{/ \aO}(\aC) \to \aC$ induced by the image of the object in
$\aO$.  Then the {\em space of $\aO$-algebra structures} on an object
$X$, denoted $\Alg_{/ \aO}(X)$, is the largest Kan complex contained
in the full subcategory of $\Alg_{/ \aO}(\aC)$ spanned by objects
which project to $X$ under this map.
\end{definition}

Given an $\infty$-operad $\O^\otimes$ and $\O$-monoidal
$\i$-categories $p\colon \aC^{\otimes}\to\O^\otimes$ and
$q\colon \aD^{\otimes}\to\O^\otimes$, we have two associated
categories of functors between them:  

\begin{enumerate} 
\item The $\i$-category of $\i$-operad maps $\Alg_{\aC/\O}(\aD)$,
which should be thought of as the analogue of lax $\O$-monoidal
functors; consult \cite[Definition~2.1.2.7]{HA} for details. 

\item The full subcategory $\Fun^{\otimes}_{\O}(\aC,\aD)$ of
$\Alg_{\C}(\aD)$ consisting of the $\i$-operad maps over $\O^\otimes$
which carry $p$-coCartesian morphisms to $q$-coCartesian morphisms,
which should be regarded as the analogue of symmetric monoidal
functors; consult \cite[Definition~2.1.3.7]{HA} for details. 
\end{enumerate}

We now explain how to produce symmetric monoidal functors of
$\i$-categories from point-set data.  We rely on the foundational work
of \cite{ElmendorfMandell2}.

\begin{lemma}\label{lem:symfunc}
Let $\aC$ and $\aD$ be symmetric monoidal categories and
$f \colon \aC \to \aD$ a lax symmetric monoidal functor.  Then there
is an induced $\i$-operad map
\[
\N(f) \colon \N(\aC)^{\otimes} \to \N(\aD)^{\otimes},
\]
making $\N(-)^{\otimes}$ into a functor from the (nerve of the)
category of symmetric monoidal categories and lax symmetric monoidal
functors to the $\i$-category of $\infty$-operads. 
\end{lemma}

\begin{proof}
Associated to any symmetric monoidal category is an underlying
multicategory, and the work of \cite[\S 3]{ElmendorfMandell2},
particularly the proof of theorem 1.1, implies that a lax map
$f \colon \aC \to \aD$ induces a map of multicategories.  It is clear
from \cite[Remark~2.1.1.7]{HA} that a map of multicategories
induces a functor $f^{\otimes} \colon \aC^{\otimes} \to \aD^{\otimes}$
over the natural projections to $\Fin_*$.  Since $f^{\otimes}$ came
from a map of multicategories, on passage to the coherent nerve
$f^{\otimes}$ takes inert morphisms to inert morphisms and so induces
a map of $\i$-operads.  The behavior of composites follows immediately
from the functoriality of the passage from lax maps to multicategory
maps.
\end{proof}

Next, we integrate this with the localization $\N(\aC)[W^{-1}]$.  

\begin{proposition}\label{cor:indufunc}
Let $\aC$ and $\aD$ be symmetric monoidal categories equipped with
subcategories $W_{\aC} \subset \aC$ and $W_{\aD} \subset \aD$ of weak
equivalences such that both collections of weak equivalences are
closed under tensor with any fixed object.  Let $f \colon \aC \to \aD$
be a lax symmetric monoidal functor such that $f$ preserves weak
equivalences and $f$ is weakly symmetric monoidal in the sense that
the maps $1_{\aD} \to f(1_{\aC})$ and $f(c_1) \otimes f(c_2) \to
f(c_1 \otimes c_2)$, for all pairs of objects $c_1$ and $c_2$ of $\aC$, are in $W_{\aD}$.  Then
 
\begin{enumerate}
\item there are symmetric monoidal $\i$-categories
$\N(\aC)[W_{\aC}^{-1}]^{\otimes}$ and $\N(\aD)[W_{\aD}^{-1}]^{\otimes}$ with
underlying $\i$-categories $\N(\aC)[W_{\aC}^{-1}]$ and
$\N(\aD)[W_{\aD}^{-1}]$ 
respectively, and 
\item there is a map of $\i$-operads 
\[
\tilde{f} \colon \N(\aC)[W_{\aC}^{-1}]^{\otimes} \to \N(\aD)[W_{\aD}^{-1}]^{\otimes} 
\]
such that on underlying $\i$-categories $\tilde{f}$ restricts to the
functor induced by $f$ via lemma~\ref{lem:symfunc}.
\end{enumerate}

\end{proposition}

\begin{proof}
By Lemma~\ref{lem:symfunc}, we have a map
\[
\xymatrix{
\N(\aC)^{\otimes} \ar[r]^{\N(f)} & \N(\aD)^{\otimes},
}
\]
and since by hypothesis the localization $N(\aD^{\otimes}) \to
N(\aD^{\otimes})[W_{\aD}^{-1}]$ is symmetric monoidal (it satisfies
the criterion of~\cite[Proposition 4.1.3.4]{HA}), we have a lax symmetric monoidal
functor 
\[
\xymatrix{
\N(\aC)^{\otimes} \ar[r]^-{\N(f)} & \N(\aD)[W_{\aD}^{-1}]^{\otimes}.
}
\]
Furthermore, $\N(f)$ lies in the subcategory
\[
\Fun^\otimes(\N(\aC),\N(\aD)[W_{\aD}^{-1}])\subseteq\Alg_{\N(\aC)}(\N(\aD)[W_{\aD}^{-1}]) 
\]
of symmetric monoidal functors since each of the comparison maps (for each active map $\mu\colon \langle n\rangle\to\langle m\rangle$ in $\Gamma$)
\[
\mu_!(F(\aA_1,\cdots,\aA_n))\to F(\mu_!(\aA_1,\cdots,A_n)))
\]
is a weak equivalence in $\aD$, and hence becomes an equivalence in
$\N(\aD)[W_{\aD}^{-1}]$~\cite[Definition~1.3.4.1]{HA}.
Finally, since $f$ takes elements of $W_{\aC}$ to elements of
$W_{\aD}$, the result follows from \cite[Proposition~4.1.3.4]{HA}, which 
we summarize in the following commutative diagram:
\[
\xymatrix{
\N(\aC)^{\otimes} \ar[r]^{\N(f)} \ar[d] & \N(\aD)^{\otimes} \ar[d] \\
\N(\aC)[W_{\aC}^{-1}]^{\otimes} \ar[r]^{\N(\tilde{f})}
& \N(\aD)[W_{\aD}^{-1}]^{\otimes}. \\ 
}
\]
\end{proof}

We also use a comparison between point-set $E_\infty$ algebras in a
symmetric monoidal combinatorial simplicial model category with
$E_\infty$ algebras in the underlying $\i$-category.  We
thank Jacob Lurie for suggesting this argument.  (See
also~\cite[\S6.2]{heuts} for progress towards this kind of result in
the context of arbitrary multicategories rather than just operads.)

\begin{proposition}\label{cor:opeinf}
Let $\aO$ be a cofibrant simplicial $E_\infty$-operad and
$\aC$ a symmetric monoidal simplicial model category.
Then there is an equivalence of $\i$-categories 
\[
\N(\Alg_{\aO}(\aC))[W^{-1}] \htp \CAlg(\N(\aC)[W^{-1}]).
\]
\end{proposition}

\begin{proof}
First, there is an equivalence of symmetric monoidal $\i$-categories
between $\N(\aC)[W^{-1}]^{\otimes}$ and $\N(\aC^{\circ})^{\otimes}$,
where $\aC^{\circ}$ denotes the cofibrant-fibrant objects in
$\aC$~\cite[Variant 4.1.3.17]{HA}.  By the functoriality of the
operadic homotopy coherent nerve, we have a canonical map
\[
\N(\Alg_{\aO}(\aC^{\circ})) \to \Alg_{\N(\aO)}(\N(\aC^{\circ})) \htp \Alg_{\N(\aO)}(\N(\aC)[W^{-1}]).
\]
Since this map clearly preserves weak equivalences, it factors as a
map 
\[
\gamma \colon \N(\Alg_{\aO}(\aC^{\circ}))[W^{-1}] \to \Alg_{\N(\aO)}(\N(\aC)[W^{-1}]).
\]
Since $\aO$ is cofibrant, the results of~\cite[\S 4]{Spitzweck} produce a
$J$-semi model category structure on the category $\Alg_{\aO}(\aC)$.
The proof of~\cite[Theorem 1.3.4.20]{HA} goes through in this context
to show that the map
\[
\N(\Alg_{\aO}(\aC^{\circ})[W^{-1}] \to \N(\Alg_{\aO}(\aC))[W^{-1}]
\]
is an equivalence of $\i$-categories.

Finally, we follow the strategy of the proof of~\cite[Theorem
4.4.4.7]{HA} to prove that $\gamma$ is an equivalence of
$\i$-categories.  We have the commutative diagram
\[
\xymatrix{
\N(\Alg_{\aO}(\aC^{\circ}))[W^{-1}] \ar[rr]^{\gamma} \ar[dr]^{G}
&& \Alg_{\N(\aO)}(\N(\aC)[W^{-1}]) \ar[dl]^{G} \\ &
\N(\aC)[W^{-1}] &
}
\]
and we will apply the $\i$-categorical Barr-Beck theorem
via~\cite[Corollary 6.2.2.14]{HA}.  The verification of the required
hypotheses proceeds exactly as in~\cite[4.4.4.7]{HA} except for two
conditions.  First, we need to check that
$G \colon \N(\Alg_{\aO}(\aC^{\circ})[W^{-1}] \to \N(\aC)[W^{-1}]$
preserves geometric realizations of simplicial objects.  Second, we
need to show that the free $\aO$-algebra functor $\aC^{\circ}\to\Alg_{\aO}(\aC^{\circ})$ on a cofibrant-fibrant object $X$ is
computed by the colimit $\coprod_n (\aO(n) \times X^n)_{h\Sigma_n}$.
Both of these can be verified using the $J$-semi model structure on
$\aO$-algebras in $\aC$ and the fact that $\aO$ is a cofibrant
$E_\infty$ operad.  For the first, the arguments for~\cite[\S
VII.3]{EKMM} apply since the free $\aO$-algebra functor commutes with
geometric realization on $\aC$.  For the second, this is an immediate
consequence of the fact that $\aO$ is cofibrant $E_\infty$ and so the
derived functor of the free $\aO$-algebra functor is a homotopy
colimit of the desired form.
\end{proof}

We conclude this section with some remarks about the behavior of the
unit object in monoidal $\i$-categories.  Let $\mathbf{1}$ denote the
unit object in an $E_n$-monoidal $\i$-category.  Our main results rely
on the following analogues of the standard fact that the unit is
initial in a monoidal category; these $\i$-categorical versions follow
from~\cite[Proposition 3.2.1.8]{HA} (e.g., see~\cite[Corollary
3.2.1.9]{HA}).

\begin{lemma}\label{lem:unit}
Let $\aC^\otimes$ be an $E_n$-monoidal $\i$-category, $n\geq 0$.
Then the $\i$-category $\Alg_{/E_n}(\C)$ has an initial object, and an
object $A$ of $\Alg_{/E_n}(\C)$ is initial if and only if the unit map
$\mathbf{1}\to A$ is an equivalence in $\C$. 
\end{lemma}

\begin{corollary}\label{cor:triv}
Let $\aC^\otimes$ be an $E_n$-monoidal $\i$-category, $n\geq 0$.  Then, the
$\i$-category $\Alg_{/E_n}(\mathbf{1})$ of $E_n$ algebra structures on
the unit object $\mathbf{1}$ of $\C$ is contractible and, for any
other $E_n$ algebra $A$ of $\C$, the space of $E_n$ algebra maps from
$\mathbf{1}$ to $A$ is contractible.
\end{corollary}

\section{Multiplicative Morita theory}\label{sec:multMor}

The $\i$-category $\idemstabcat$ of idempotent-complete small stable
$\i$-categories has a symmetric monoidal structure with product
$\otimes^{\vee}$ and the unit the $\i$-category $\ispec^{\omega}$ of
compact objects in $\ispec$~\cite[\S 6.3.1]{HA}.
Recall from~\cite[Thms.~4.22 and 4.23]{BGT} that we have a description
of $\idemstabcat$ as the accessible localization of the $\i$-category
$\N(\Spcat)[W^{-1}]$ along the Morita equivalences.  The goal of this
section is to promote this equivalence to an equivalence of symmetric
monoidal $\i$-categories, using the smash product of spectral
categories; see Theorem \ref{thm:symcomp}.

The category $\Spcat$ has a closed symmetric monoidal product given by
taking $(\aC, \aD)$ to the spectral category with objects $\ob \aC
\times \ob \aD$ and morphism spectra $\aC(c,c') \sma \aD(d,d')$.
However, the smash product of cofibrant spectral categories is not
necessarily cofibrant, and consequently the model structure is not monoidal \cite{Spectral} (and see
\cite{Toen} for a discussion of this in the setting of
DG-categories).  This issue is one of the persistent technical
difficulties in working with these models of $\idemstabcat$. 

To resolve this problem, we employ the notion of flat objects and
functors (e.g., see \cite[B.4]{hhr}).  Recall that a functor between
model categories is flat if it preserves weak equivalences and
colimits.  An object $X$ of a model category (whose underlying
category is monoidal with respect to a tensor product $\otimes$) is
then said to be flat if the functor $X \otimes (-)$ is a flat functor.
Cofibrant objects in a monoidal model category are flat; in
particular, cofibrant spectra are flat.  The utility of this
definition comes from the fact that the smash product of flat spectra
computes the derived smash product.

We define a spectral category $\aC$ to be pointwise-cofibrant if each
morphism spectrum $\aC(x,y)$ is a cofibrant spectrum.  The following
proposition summarizes the facts about pointwise-cofibrant spectral
categories that we will need.  Recall that a spectral category $\aC$
has an associated spectral category of perfect modules (equivalently,
homotopically compact modules, or retracts of finite cell modules),
and that a map of spectral categories $f\colon \aC\to\aD$ is a {\em Morita
equivalence} if $f$ induces a DK-equivalence between spectral
categories of perfect modules; see \cite[\S 2]{BGT} for details.

\begin{proposition}\label{prop:flatprops}
\hspace{5 pt}
\begin{enumerate}
\item Every spectral category is functorially Morita equivalent to a
  pointwise-cofibrant spectral category with the same objects.
\item The subcategory of pointwise-cofibrant spectral categories is closed
  under the smash product.
\item A pointwise-cofibrant spectral category is flat with respect to the
  smash product of spectral categories.
\item If $\aC$ and $\aD$ are pointwise-cofibrant spectral categories, the
  smash product $\aC \sma \aD$ computes the derived smash product $\aC
  \sma^{\L} \aD$. 
\end{enumerate} 
\end{proposition}

\begin{proof}
By construction of the generating cofibrations
(see \cite[Def.~4.4]{Spectral}), there exists a cofibrant resolution
functor $Q(-)$ on $\Spcat$ such that for any spectral category $\cC$
the spectral functor $Q(\cC) \to \cC$ induces the identity map on the
set of objects. Item (1) follows then
from \cite[Prop.~4.18]{Spectral}, which shows that every cofibrant
spectral category is pointwise-cofibrant. Item (2) follows from the
fact that the smash product of cofibrant spectra remains
cofibrant. Item (4) follows from item (3).  

Let $\cC$ be a pointwise-cofibrant spectral category. The functor
$-\wedge\cC \colon \Spcat \to \Spcat$ clearly preserves colimits as the
symmetric monoidal structure is closed. Hence, in order to prove item
(3), it remains to show that if $f:\cA \to \cB$ is a Morita equivalence
(see \cite[Def.~6.1]{IMRN}), then
$f\wedge \id: \cA \wedge \cC \to \cB \wedge \cC$ is also a Morita
equivalence.

If $f\colon\cA\to\cB$ is a Morita equivalence then the induced map
$f_!\colon\Mod(\cA)\to\Mod(\cB)$ is a DK-equivalence of spectral
categories.  Let $\aC$ be a pointwise-flat spectral category.  We must
show that
\[
(f\sma\id)_! \colon \Mod(\cA\sma\cC)\to\Mod(\cB\sma\cC)
\]
is also a DK-equivalence, which is to say that it is homotopically fully faithful and essentially surjective.
We begin with the former.
Since any cofibrant $\aA\sma\aC$-module is a retract of a cellular
$\aA\sma\aC$-module and $(f\sma\id)_!$ is a left Quillen functor which
preserves representable modules, it suffices to check that
$f \sma \id$ itself is homotopically fully faithful. 
But this follows because $\cC$ is pointwise-cofibrant, so
\begin{equation*}
(f\sma\id)((a',c'),(a,c)) \colon (\cA\wedge \cC)((a',c'),(a,c)) \too (\cB \wedge \cC)((fa',c'),(fa,c))\,
\end{equation*}
is a weak equivalence of spectra for any pair of objects $(a,c)$ and $(a',c')$ of $\cA\sma\cC$.
To verify essential surjectivity, it suffices to check that the $\cB\sma\cC$-module $\widehat{(b,c)}$ represented by the object $(b,c)\in\cB\sma\cC$ is equivalent to the image of an $\cA\sma\cC$-module.
Since $f$ is a Morita equivalence, there exists a perfect $\cA$-module $M$ and a weak equivalence of $\cB$-modules $f_!M\simeq\widehat{b}$.
Thus
\[
(f\sma\id)_!(M\sma\widehat{c})(b',c')\simeq M(b')\sma\cC(c',c)\simeq\cB(b',b)\sma\cC(c',c)\simeq (f\sma\id)_!(\widehat{(b,c)})(b',c'),
\]
so that $(f\sma\id)_!$ is also homotopically essentially surjective.
\end{proof}

\begin{remark}
The use of pointwise-cofibrant spectral categories is analogous to the
use of homotopically flat DG-modules in the differential graded
setting.  See for instance \cite{CisTab} for a similar development to
the work of this section in that context.
\end{remark}

Therefore, we use the subcategory $\Spcat^{\sflat}$ of pointwise-cofibrant
spectral categories to produce a suitable symmetric monoidal model of
the $\i$-category of idempotent-complete small stable
$\i$-categories.  The following lemma is a first consistency check:

\begin{lemma}
Let $\Spcat^{\cc}$ denote the full subcategory of cofibrant objects in
$\Spcat$.  The functor induced by cofibrant replacement
$\Spcat^{\sflat} \to
\Spcat^{\cc}$ induces a categorical equivalence
\[
\N(\Spcat^{\sflat})[\aW^{-1}] \to \N(\Spcat^{\cc})[\aW^{-1}].
\]
\end{lemma}

\begin{proof}
Recall from the proof of Proposition~\ref{prop:flatprops} that the
cofibrant replacement of a spectral category is pointwise-cofibrant.
It is now clear that the inclusion and the functorial cofibrant
replacement induce inverse equivalences.
\end{proof}

By combining Proposition~\ref{prop:consist} with \cite[4.22,
4.23]{BGT} we obtain:

\begin{proposition}\label{prop:compsym}
There is an equivalence of $\i$-categories
\[
\idemstabcat\simeq \N(\Spcat^{\sflat})[\aW^{-1}].
\]
\end{proposition}

Furthermore, Proposition~\ref{prop:flatprops} implies that
$\Spcat^{\sflat}$ is a symmetric monoidal category such that the smash
product preserves equivalences.  Thus we conclude from
\cite[Proposition~4.1.3.4 and Example~4.1.3.6]{HA} that we obtain a
symmetric monoidal $\i$-category
$(\N(\Spcat^{\sflat})[\aW^{-1}])^{\otimes}$ with underlying
$\i$-category $\N(\Spcat^{\sflat})[\aW^{-1}]$.  We now upgrade the
comparison of \cite[4.22, 4.23]{BGT} to a comparison of symmetric
monoidal $\i$-categories.  

To explain how to do this, we need to review some details about the
construction of the symmetric monoidal structure on $\idemstabcat$.
Recall that this is determined by the symmetric monoidal structure on
the $\i$-category $\stabprcat$, as follows.  First, there is an
equivalence between the $\i$-category $\stabcgprcat$ of
compactly-generated stable presentable $\i$-categories and
idempotent-complete small stable $\i$-categories which is realized by
passage to compact objects (denoted $(-)^{\omega}$) and the formation
$\Ind(-)$ of the Ind-category~\cite[\S 5.5.7]{HTT}.  Next,
$\stabcgprcat$ inherits a symmetric monoidal structure from the
structure on $\stabprcat$~\cite[6.3.7.11]{HA}.  Unwinding this, for
$\aC$ and $\aD$ in $\idemstabcat$, the tensor product
$\aC \otimes^{\vee} \aD$ can be computed as
$(\Ind(\aC) \otimes \Ind(\aD))^{\omega}$.

In turn, the symmetric monoidal structure on $\stabprcat$ can be
obtained as a symmetric monoidal localization of the $\i$-category
$\prcat$ of presentable $\i$-categories, as we now describe.  The
symmetric monoidal structure on $\prcat$ is induced from the monoidal
structure on the $\i$-category of $\i$-categories~\cite[6.3.1.14]{HA},
and in $\prcat$ the $\i$-category of spectra $\ispec$ is an idempotent
object~\cite[6.3.2.18]{HA}.  Therefore, the functor $- \otimes \ispec$
in $\prcat$ defines a symmetric monoidal localization on $\prcat$ such that the full subcategory
of local objects is precisely $\stabprcat$~\cite[6.3.2.19]{HA}.

Translating back to small $\i$-categories, the symmetric monoidal
structure on $\prcat$ is determined by the Cartesian monoidal
structure on $\icat$.  Furthermore, we can characterize this structure
in terms of the category $\Scat$ of small simplicial categories.
Since $\Scat$ has all homotopy limits, the associated $\i$-category
$\N(\Scat)[W^{-1}]$ admits a (necessarily unique) Cartesian symmetric monoidal structure
which we denote $\N(\Scat)[W^{-1}]^{\otimes}$;
see \cite[Corollary~2.4.1.9]{HA}.  The uniqueness implies that there
is an equivalence of symmetric monoidal $\i$-categories
\[
\N(\Scat)[W^{-1}]^{\otimes} \htp \icat^{\times}
\]
in which $\icat$ is endowed with the cartesian symmetric monoidal
structure.

Recall that $\ispec$ is an idempotent commutative algebra object of $\prcat$ and that the presheaf functor $\Pre:\icat\too\prcat_\omega$ (specifically, the covariant version in which the functor $\Pre(\aC)\to\Pre(\aD)$ induced by a functor $\aC\to\aD$ is left adjoint to the restriction $\Pre(\aD)\to\Pre(\aC)$) is symmetric monoidal.
We therefore obtain a symmetric monoidal functor
\[
\Pre_{\ispec}\simeq\Pre(-)\otimes\ispec\colon \icat\too\prcat_{\mathrm{st},\omega},
\]
the functor which sends the small $\i$-category $\aC$ to the compactly generated $\i$-category of spectral presheaves on $\aC$.
Composing with the symmetric monoidal equivalence $(-)^\omega\colon \prcat_{\mathrm{st},\omega}\to\idemstabcat$, we obtain a symmetric monoidal functor
\[
\Pre_{\ispec}(-)^\omega\colon \icat\too\idemstabcat.
\]
\begin{lemma}
There is a unique commutative algebra structure
${\idemstabcat}^\otimes$ on $\idemstabcat$ in $\prcat$ such that 
\[
\Pre_{\ispec}(-)^{\omega} \colon \icat \to \idemstabcat
\]
extends to a map of commutative algebras
$\Cat_\i^\times\to{\idemstabcat}^\otimes$.
\end{lemma}

\begin{proof}
Recall that a commutative algebra in $\prcat$ is a symmetric monoidal
$\i$-category such that tensoring with a fixed object preserves
colimits.  As a consequence, it suffices to show that the functor
$\Pre_{\ispec}(-)^{\omega}$ uniquely determines the tensor products of
a collection of objects that generate $\idemstabcat$ under colimits.

As reviewed in Section~\ref{sec:background}, $\idemstabcat$ is
equivalent to the $\i$-category produced by taking the underlying
$\i$-category associated to the model category of spectral categories
with weak equivalences the DK-equivalences and Bousfield localizing at
the Morita equivalences.  This model category is cofibrantly
generated, with generating cofibrations and acyclic cofibrations
(described explicitly in~\cite[\S 2.16]{BergerMoerdijk}) that have
mapping spectra determined by the generating acyclic cofibrations and
acyclic cofibrations of the category of symmetric spectra.  In
particular, the generating cofibrations and acyclic cofibrations have
mapping spectra that are suspension spectra and therefore are in the
image of $\Sigma^{\infty}_+$ applied to the category of simplicial
categories.

Finally, since any object in the Bousfield localization is weakly
equivalent to a cellular object (generated as a filtered colimit of
pushouts along generating cofibrations) in the DK-equivalence model
structure, we see that any object of $\idemstabcat$ is weakly
equivalent to a colimits of objects in the image of
$\Pre_{\ispec}(-)^{\omega}$.  The result now follows from the fact
that $\Pre_{\ispec}(-)^{\omega}$ is strong symmetric monoidal. 
\end{proof}

We finally have all the tools needed for the proof of the
multiplicative Morita theory result.

\begin{theorem}(Multiplicative Morita theory)\label{thm:symcomp}
There is an equivalence of symmetric monoidal $\i$-categories 
\[
(\idemstabcat)^{\otimes}\simeq(\N(\Spcat^{\sflat})[\aW^{-1}])^{\otimes}.
\]
\end{theorem}

\begin{proof}
Consider the composite 
\[
\xymatrix{
\Phi \colon \Scat^{\f} \ar[r]^-{(-)_+} &
(\Scat^{\f})_* \ar[r]^-{\Sigma^{\infty}} & \Spcat^{\sflat}, \\ 
}
\]
where $\Scat^{\f}$ denotes the full subcategory of fibrant simplicial
categories and $(\Scat)_*$ denotes the pointed simplicial categories
(which means that all the mapping complexes have basepoints).  The
functor $\Phi$ satisfies the requirements of
Proposition~\ref{cor:indufunc}, and so we have an induced map of
symmetric monoidal $\i$-categories   
\[
\xymatrix{
\icat^{\times} \htp \N(\Scat^{\f})[W^{-1}]^{\times} \ar[r]^-{\Sigma^{\infty}_+}
& \N(\Spcat^{\sflat})[\widetilde{W}^{-1}]^{\otimes}, \\
}
\]
where here $\widetilde{W}$ denotes the class of DK-equivalences of
spectral categories.  Next, by Proposition~\ref{prop:flatprops},
composing with the symmetric monoidal localization at the Morita
equivalences of spectral categories gives rise to a map of symmetric
monoidal $\i$-categories
\[
\xymatrix{ 
\theta \colon \icat^{\times} \ar[r]^-{\Sigma^\infty_+}
& \N(\Spcat^{\sflat})[\widetilde{W}^{-1}]^\otimes \ar[r] & 
\N(\Spcat^{\sflat})[\aW^{-1}])^{\otimes}.
}
\]
(Recall that we know from~\cite[4.23]{BGT} that this localization is
in fact the Bousfield localization at a set of generating Morita
equivalences.)
By Proposition~\ref{prop:compsym}, we have that
$\N(\Spcat^{\sflat})[\aW^{-1}]$ is equivalent to $\idemstabcat$.
Therefore, by the discussion preceding the theorem, in order to
identify the symmetric monoidal structure on
$\N(\Spcat^{\sflat})[\aW^{-1}]$ as a model of
$(\idemstabcat)^{\otimes}$, it suffices to identify the composite
$\theta$ as $\Pre_{\ispec}(-)^{\omega}$.

The comparison of \cite[4.22, 4.23]{BGT} identifies the functor
\[
\Psi \colon \SpcatM^{\sflat} \to \SpcatM^{\sflat},
\]
that takes a small pointwise-cofibrant spectral category $\aC$ to the
cofibrant-fibrant homotopically compact objects in the projective
model structure on $\Mod(\aC)$, as the localization at the Morita
equivalences.  Therefore, ignoring the symmetric monoidal structure,
the underlying functor of $\theta$ can be described as the composite
of $\Phi$, the equivalence between $\Spcat$ and $\SpcatM$, and $\Psi$.
But now the identification of $\theta$ is clear, since
\[
\N(\Psi(\Sigma^\i_+\aC))\simeq\N(\Mod(\Sigma^\i_+\aC))^\omega\simeq\N(\Fun(\aC^{\op},\aS_\infty))^\omega
\]
is precisely a model of $\Fun(\N(\aC)^{\op}, \ispec)^{\omega}\simeq\Pre_{\ispec}(\N(\aC))^\omega$.
\end{proof}

One immediate consequence of the preceding comparison result is that
we can explicitly describe the mapping spectra in the tensor product
of small stable $\i$-categories in terms of the smash product of
spectra.  Specifically, let $\aC$ and $\aD$ be small stable idempotent
complete $\i$-categories and $\widetilde{\aC}$ and $\widetilde{\aD}$
cofibrant-fibrant pre-triangulated spectral categories lifting $\aC$
and $\aD$.  Then there is a natural equivalence
\[
(\aC \otimes^{\vee} \aD)((c,d),(c',d')) \htp \widetilde{\aC}(c,c') \sma \widetilde{\aD}(d,d').
\]

\section{The symmetric monoidal structure on noncommutative motives}\label{sec:mon}

In this section, we show that $\Motadd$ and $\Motloc$ are symmetric
monoidal $\i$-categories and that the localization functors
\[
\Umot \colon \idemstabcat \to \Motadd \qquad \textrm{and} \qquad \Uloc \colon \idemstabcat \to \Motloc
\]
are symmetric monoidal.  This is an interesting result in its own
right; for instance, it implies that $\Motadd$ is canonically enriched
in $A(*) = K(\mathbb{S})$-modules (see Corollary~\ref{cor:Aenriched}).
Herein, we use these results to compare symmetric monoidal structures
on the $\i$-categories of colimit-preserving functors from $\Motadd$
to $\ispec$ and additive functors from $\idemstabcat$ to $\ispec$ (and
analogously in the localizing case).

Our approach is motivated by the following classical picture.  If $R$
is a spectrum, then we may view the associated cohomology theory as a
(pre)sheaf of spectra $\cT^{\op}\to\aS_\infty$ on the $\i$-category of
spaces.  If $R$ has an $E_\i$ structure, then this functor is
canonically lax symmetric monoidal via the external cup product
pairing
\[
\xymatrix{
\Map(\Sigma^\i_+ X,R)\land\Map(\Sigma^\i_+Y,R) \ar[r] & \Map(\Sigma^\i_+
(X\times Y),R\land R) \ar[d] \\ \ & \Map(\Sigma^\i_+ (X\times Y), R).
}
\]
In fact, the lax symmetric monoidal structure on
$\Map(\Sigma^\i_+(-),R)$ is equivalent to a symmetric monoidal
structure on $R$ itself, as the latter may be recovered by restricting
to the point.  Here, we are studying the analogous picture in the
setting of noncommutative motives.

We begin by fixing some notation.  We will refer to commutative
algebra objects of $\stabprcat$ as {\it stable presentable symmetric
monoidal $\i$-categories} and denote by $\symstabprcat$ the
$\i$-category on these objects and morphisms the colimit-preserving
symmetric monoidal functors.  Note that this specifies a full
subcategory of the more general class of (large) symmetric monoidal
$\i$-categories: an object of the latter is an object of the former if
and only if the category is stable, presentable, and the monoidal
product commutes with colimits in each variable; see \cite[\S
6.3.2]{HA}.  Also, recall that $\idemstabcat$ is itself presentable
and that the monoidal product commutes with colimits in each variable.

Now, recall that $\Umot$ is defined
to be the composite of the Yoneda embedding 
\[
\phi \colon \idemstabcat \to \Pre((\idemstabcat)^{\omega})
\]
(where the presheaves are restricted to the full subcategory
$(\idemstabcat)^{\omega}$ of $\idemstabcat$ consisting of the compact
objects), followed by stabilization and then localization at a
generating set $\aE$ of split-exact sequences.  The functor $\Uloc$ is
defined analogously, with localization taken with respect to a
generating set of all exact sequences.  As such, our investigation of
the symmetric monoidal structure involves assembling the analysis of
each piece of the composite.

\subsection*{Yoneda embedding and stabilization}
First, we observe that the symmetric monoidal structure on
$\idemstabcat$ descends to the subcategory of compact objects.   For
this, we need a technical proposition.

\begin{proposition}\label{compactclosed}
Let $\aA$ be a compact object of $\idemstabcat$ and $\{\aB_i\}_{i\in
I}$ a filtered diagram in $\idemstabcat$.  Then the map
\[
\colim_{i\in I}\Fun^{\ex}(\aA,\aB_i)\to\Fun^{\ex}(\aA,\colim_{i\in
I}\aB_i)
\]
is an equivalence of small stable idempotent-complete $\i$-categories.
\end{proposition}

\begin{proof}
The inclusion
$\Fun^{\ex}(\A,\B)\to\Fun^{\ex}(\A\idemtimes\B^{\op},\ispec)$ is the full
subcategory on those $f:\A\idemtimes\B^{\op}\to\ispec$ which restrict to
representable functors for each $a$ in $A$.  This gives a commuting
square
\[
\xymatrix{
\colim\Fun^{\ex}(\A,\B_i)\ar[r]\ar[d] & \colim\Fun^{\ex}(\A\idemtimes\B_i^{\op},\ispec)\ar[d]\\
\Fun^{\ex}(\A,\colim\B_i)\ar[r] & \Fun^{\ex}(\A\idemtimes\colim\B_i^{\op},\ispec)},
\]
where we use the fact that $(\colim B_i)^{\op} \htp \colim
B_i^{\op}$.  The horizontal maps are fully faithful inclusions and the
vertical map on the right is an equivalence
since $\Fun^{\ex}(\A\idemtimes(-)^{\op},\ispec)$ preserves filtered colimits.
It follows that the vertical map on the left is fully faithful.
Restricting to maximal subgroupoids (i.e., the mapping spaces), we see
that for $A$ compact the left-hand vertical map induces an equivalence
(by definition), so it is also essentially surjective.
\end{proof}

This allows us to prove the following result.

\begin{proposition}\label{compactsymmon}
The $\i$-category of compact objects $(\idemstabcat)^\omega$ in
$\idemstabcat$ is a full symmetric monoidal subcategory. 
\end{proposition}

\begin{proof}
Let $\aA$ and $\aB$ be compact small stable
idempotent-complete $\i$-categories.
We must show that $\aA\otimes^\vee\aB$ is compact.
To this end, let $\aC_i$ be a filtered system of small stable
idempotent-complete $\i$-categories.
By proposition \ref{compactclosed}, we obtain the following sequence
of equivalences
\[
\begin{split}
\map(\aA\otimes^\vee\aB,\colim_{i \in
  I}\aC_i)&\simeq\map(\aA,\Fun^{\ex}(\aB,\colim_{i \in
  I}\aC_i))\\&\simeq\map(\aA,\colim_{i \in I}\Fun^{\ex}(\aB,\aC_i)) \\
&\simeq\colim_{i \in
  I}\map(\aA,\Fun^{\ex}(\aB,\aC_i)) \simeq\map(\aA\otimes^\vee\aB,\aC_i),
\end{split}
\]
which implies that $\map(\aA\otimes^\vee\aB,-)$ commutes with filtered colimits.
(Here we are denoting by $\map(-,-)$ the derived simplicial mapping
space.)
\end{proof}

Next, we use the fact that passage to presheaves and stabilization are
both symmetric monoidal functors.

\begin{proposition}
The stable presentable $\i$-category
\[
\Stab(\Pre(\idemstabcat)^{\omega}) \htp \Fun(((\idemstabcat)^{\omega})^{\op},\ispec)
\]
has a canonical presentable symmetric monoidal structure.
This structure is compatible with the symmetric monoidal structure on 
$\idemstabcat$ in the sense that the functor 
\[
\idemstabcat \to \Stab(\Pre(\idemstabcat)^{\omega})
\]
is symmetric monoidal.
\end{proposition}

\begin{proof}
The $\i$-category $\Pre(\aC)$ admits a symmetric monoidal structure
such that the Yoneda embedding $\aC \to \Pre(\aC)$ is a symmetric
monoidal functor; see~\cite[Corollary~6.3.1.12]{HA}.  Furthermore, the
$\i$-category $\Stab(\aC)$ admits a symmetric monoidal structure such
that the stabilization functor $\aC \to \Stab(\aC)$ is symmetric
monoidal; see~\cite[Example 6.3.1.22 and Proposition 6.3.2.18]{HA}.
It follows that $\Stab(\Pre(\idemstabcat)^{\omega})$ is a stable presentable
symmetric monoidal $\i$-category.  Finally, the functor
\[
\idemstabcat \to \Stab(\Pre(\idemstabcat)^{\omega})
\]
is symmetric monoidal since we know that $\idemstabcat$ is compactly 
generated; it is generated under filtered colimits by
$(\idemstabcat)^{\omega}$ \cite[3.22]{BGT} (and $\Ind(\aC)$ is a symmetric
monoidal subcategory of $\Pre(\aC)$~\cite[Proposition~6.3.1.10]{HA}).

\end{proof}

\subsection*{Localization at a generating set}

Finally, in order to show that $\Motadd$ and $\Motloc$ are symmetric
monoidal, it will suffice to show that the localization at $\aE$
inherits the structure of a symmetric monoidal $\i$-category.  As a
consequence of the argument, we will also show that the localization
functor is symmetric monoidal.  Recall that $\aE$ consists of the maps
\[
\widehat{\aB} / \widehat{\aA} \to \widehat{\aC}
\]
in $\Stab(\Pre(\idemstabcat)^{\omega})$ associated to a
generating set of split-exact sequences $\aA \to \aB \to \aC$ in
$(\idemstabcat)^{\omega}$, where here $\widehat{\aB} / \widehat{\aA}$
denotes the cofiber~\cite[\S 5]{BGT}.  

Provided that we can show that the localization functor is compatible
with the symmetric monoidal structure (in the sense
of~\cite[Definition~2.2.1.6]{HA}), then~\cite[Proposition~2.2.1.9]{HA}
will establish that the localization is symmetric monoidal.  Recall
that to show compatibility, by~\cite[Example~2.2.1.7]{HA} it suffices
to show that for every local equivalence $X \to Y$ and any object $Z$,
the induced map $X \otimes Z \to Y \otimes Z$ is a local equivalence.
To do this, we characterize exact sequences of $\i$-categories in
terms of acyclics.  (See also~\cite[\S 7]{BM} for similar
identifications in the language of spectral categories.)  Recall that
an exact sequence is in particular a cofiber sequence~\cite[5.9]{BGT}.

\begin{lemma}\label{lem:im=ker}
Let $\aA\to\aB\to\aC$ be a cofiber sequence of idempotent-complete
small stable $\i$-categories such that $\aA\to\aB$ is fully faithful.
Then $\aA$ is canonically equivalent to the fiber (over $0$) of
$\aB\to\aC$, i.e., the full subcategory of $\aB$ on the objects which
are equivalent to $*$ in $\aC$.
\end{lemma}

\begin{proof}
Write $\aA'$ for the fiber of $\aB\to\aC$ and let $\aA\to\aA'$ be the
resulting map, which is necessarily fully faithful.  To check that it
is essentially surjective, it suffices to check on triangulated
homotopy categories.  Let $b$ be an object of $\aA'$.
Recall that we can identify $\Ho(\aC) \htp \Ho(\aB/\aA)$ as the
Verdier quotient of triangulated categories $\Ho(\aB)
/ \Ho(\aA)$~\cite[5.13]{BGT}.
By the description of maps in the Verdier quotient, we can easily
check that the identity map $b\to b$ must be equal to the zero map
$b\to b$.  That is, there is a commutative diagram
\[
\xymatrix{
& b\ar[rd]^{\id}\ar[ld] &\\ b & 0\ar[u]\ar[l]\ar[r]\ar[d] & b\\ &
b\ar[ur]_{0}\ar[ul] &}
\]
which implies that the cofiber $b$ of $0\to b$ must lie in $\aA$.
\end{proof}

\begin{lemma}\label{lem:tensorexact}
Let $\aA\to\aB\to\aC$ be an exact sequence of idempotent-complete small stable $\i$-categories.
Then, for any idempotent-complete small stable $\i$-category $\aD$, the sequence
\[
\aD\otimes\A\too\aD\otimes\aB\too\aD\otimes\aC
\]
is exact.
\end{lemma}

\begin{proof}
Since the tensor product on $\idemstabcat$ commutes with colimits, it
is enough to show that $\aD\otimes\aA$ is equivalent to the full
subcategory $\aF$ of $\aD\otimes\aB$ consisting of those objects which
are sent to zero objects in $\aD\otimes\aC$.  We first show that the
functor $\aD\otimes\aA\to\aD\otimes\aB$ is fully faithful by direct
computation of the mapping spectra: given a pair of objects $d,d'$ of
$\aD$ and $a,a'$ of $\aA$ with images $b,b'$ in $\aB$, we have that
\begin{align*}
\Map((d,a),(d',a'))\simeq&\Map(d,d')\land\Map(a,a')\simeq\\
&\Map(d,d')\land\Map(b,b')\simeq\Map((d,b),(d',b')).
\end{align*}
Since $\aD\otimes\aA$ is stable, it follows that the inclusion $\aD\otimes\aA\to\aD\otimes\aB$ is fully faithful on all objects which are retracts of finite colimits of objects of the form $(d,a)$; since all objects of $\aD\otimes\aA$ are of this form, we see that the inclusion is fully faithful.
The fact that $\aD\otimes\aA$ surjects onto $\aF$ now follows from lemma \ref{lem:im=ker}.
\end{proof} 

\begin{proposition}\label{prop:compat}
The localization of
$\Stab(\Pre(\idemstabcat)^{\omega})$ at $\aE$ is
compatible with the symmetric monoidal structure given above.
\end{proposition}

\begin{proof}
Since the $\i$-category
$\Stab(\Pre(\idemstabcat)^{\omega})$ is generated
under filtered colimits by representables and these tensors commute
with filtered colimits, it suffices to check that 
\[
\map(\widehat{\aD}\otimes(\widehat{\aC}/(\widehat{\aB}/\widehat{\aA})),F)\simeq*
\]
for all $\aD$ and split-exact sequences $\aA\to\aB\to\aC$ in $\aE$.
This follows because
\[
\widehat{\aD}\otimes(\widehat{\aC}/(\widehat{\aB}/\widehat{\aA}))\simeq\widehat{\aD\otimes \aC}/(\widehat{\aD\otimes \aB}/\widehat{\aD\otimes \aA}),
\]
and, by lemma \ref{lem:tensorexact}, $\aD\otimes \aA\to \aD\otimes \aB\to \aD\otimes \aC$ is
exact when $\aA\to \aB\to \aC$ is exact, and therefore split-exact when $\aA\to\aB\to\aC$ is split-exact.
\end{proof}

Similarly, we have the following result for localizing invariants:

\begin{proposition}\label{JLoc}
The localization of $\Stab(\Pre(\idemstabcat)^{\omega})$ at the set 
generated by those objects of the form
$\widehat{\aC}/(\widehat{\aB}/\widehat{\aA})$ for each (equivalence
class of) exact sequence $\aA\to \aB\to \aC$ with $\aB$
$\kappa$-compact is compatible with the symmetric monoidal structure
given above.
\end{proposition}

\begin{proof}
As in the proof of \ref{prop:compat}, it is enough to check that
$\map(\widehat{\aD}\otimes(\widehat{\aC}/(\widehat{\aB}/\widehat{\aA})),F)\simeq 
*$ for all $\aD$ and exact sequences $\aA\to\aB\to\aC$ with $\aB$
$\kappa$-compact.  The result follows because 
\[
\widehat{\aD}\otimes(\widehat{\aC}/(\widehat{\aB}/\widehat{\aA}))\simeq\widehat{\aD\otimes \aC}/(\widehat{\aD\otimes \aB}/\widehat{\aD\otimes \aA}), 
\]
and $\aD\otimes \aA\to \aD\otimes \aB\to \aD\otimes \aC$ is exact
since $\aA\to \aB\to \aC$ is exact by lemma \ref{lem:tensorexact}.
\end{proof}

Summarizing, we have the following theorem:

\begin{theorem}\label{thm:sym}
The $\i$-categories $\Motadd$ and $\Motloc$ are endowed with natural
symmetric monoidal structures making the functors $\Umot$ and
$\Uloc$ symmetric monoidal.  The tensor units are
$\Umot(\ispec^{\omega})$ and $\Uloc(\ispec^{\omega})$ respectively.
\end{theorem}

In particular, this implies that algebraic $K$-theory is a lax
symmetric monoidal functor.  

\begin{proposition}\label{prop:kmonoidal}
The functors
\[
K(-) = \Map(\Umot(\ispec^{\omega}),\Umot(-))) \colon \idemstabcat \to \ispec
\]
and 
\[
\bbK(-) = \Map(\Uloc(\ispec^{\omega}),\Uloc(-))) \colon \idemstabcat \to \ispec
\]
are lax symmetric monoidal.
\end{proposition}

\begin{proof}
We give the argument for $K(-)$; the proof for $\bbK(-)$ is the same.
Our work in \cite[\S 4]{BGT} gave a construction for any stable
$\i$-category of the mapping spectrum functor $\Map_{\aC}(X,-)$.
Since $\Umot$ is lax symmetric monoidal, it will suffice to show that
$\Map(\Umot(\ispec^{\omega}),-)$ is lax symmetric monoidal.  First,
observe that this functor preserves limits, and so by the adjoint
functor theorem \cite[5.5.2.9]{HTT} it has a left adjoint.  The left
adjoint can be described as follows.  Since $\Motadd$ is a presentable stable
$\i$-category, it is tensored over $\ispec$
(see \cite[Remark~6.3.2.17]{HA}) in the sense of~\cite[Definition
4.2.1.19]{HA}.  Therefore, the left adjoint is given by
$\Umot(\ispec^{\omega}) \otimes (-)$.  By definition, this left
adjoint is symmetric monoidal.
Then,~\cite[Corollary~8.3.2.7]{HA} implies (as 
in \cite[Example~8.3.2.8]{HA}) that there exists a lax symmetric
monoidal right adjoint extending $\Map(\Umot(\ispec^{\omega}),-)$.
\end{proof}

This result in turn has the following corollary: 

\begin{corollary}\label{cor:kmon}
Let $\aA$ be an $E_n$ object in the symmetric monoidal $\i$-category
$\idemstabcat$, $0 \leq n\leq\infty$.  Then $K(\aA)$ and $\bbK(\aA)$
are $E_n$ ring spectra.
\end{corollary}

Since an $E_n$ ring spectrum has an $\i$-category of compact modules
which is an $E_{n-1}$ object in
$\idemstabcat$~(see \cite[Theorem~5.1.4.2]{HA}), we can specialize
corollary~\ref{cor:kmon} to conclude that algebraic $K$-theory takes
$E_n$ ring spectra to $E_{n-1}$ ring spectra.

\begin{corollary}
Let $R$ be an $E_n$ ring spectrum.  Then $K(R)$ and $\bbK(R)$ are
$E_{n-1}$ ring spectra.
\end{corollary}

\subsection*{Symmetric monoidal structure on additive and localizing invariants}\label{s:symaddinv}

We now obtain ``dual'' symmetric monoidal structures on the functor
$\i$-categories of additive and localizing invariants. 

\begin{proposition}
Let $\C$ be a small symmetric monoidal $\i$-category.  Then there is
a natural equivalence 
\[
\Fun^\L(\Stab(\Pre(\C)),\cS_\infty) \htp \Stab(\Pre(\C^{\op}))
\]
between the dual of $\Stab(\Pre(\C))$, the stable presentable
symmetric monoidal $\i$-category generated by $\aC$, and the stable
presentable symmetric monoidal $\i$-category generated by
$\aC^{\op}$.  
\end{proposition}

\begin{proof}
This follows from a calculation: If
$\Stab(\Pre(\C))=\Fun(\C^{\op},\cS_\infty)$ is the presentable stable
$\i$-category freely generated by the small $\i$-category $\C$, then 
\[
\Fun^\L(\Stab(\Pre(\C)),\cS_\infty)\simeq\Fun^\L(\Pre(\C),\cS_\infty)\simeq\Fun(\C,\cS_\infty)\simeq\Stab(\Pre(\C^{\op}))
\]
is a presentable stable $\i$-category which is dual to
$\Stab(\Pre(\C))$ under the symmetric monoidal structure on
$\stabprcat$.  (See~\cite[\S 3.3]{BGT} for further discussion.)
\end{proof}

In particular, we have the following:

\begin{corollary}\label{cor:dualmon}
Let $\C$ be a small symmetric monoidal $\i$-category.  Then the
$\i$-category $\Fun^\L(\Stab(\Pre(\C)), \cS_\infty)$ is symmetric monoidal.
\end{corollary}

\begin{proof}
If $\C$ is a small symmetric monoidal $\i$-category, then $\C^{\op}$
is also a small symmetric monoidal $\i$-category.
Therefore,~\cite[Corollary 6.3.1.12]{HA} shows that the Day
convolution product endows $\Pre(\C^{\op})$ with a canonical symmetric
monoidal structure.  Since stabilization is a symmetric monoidal
functor, this implies that $\Stab(\Pre(\C^{\op}))$ is also symmetric
monoidal.
\end{proof}

Proposition~\ref{prop:compat} and Corollary~\ref{cor:dualmon} now
imply the $\i$-categories $\Fun^\L(\Motadd, \ispec)$ and $\Fun^\L(\Motloc,
\ispec)$ are symmetric monoidal. Analogous results for the categories
of additive and localizing invariants hold:

\begin{theorem}\label{thm:bodymaintwo}
The $\i$-categories $\Fun_\add(\idemstabcat,\cS_\infty)$ and
$\Fun_\loc(\idemstabcat,\cS_\infty)$ are symmetric monoidal
$\i$-categories.  The units are the connective and non-connective
algebraic $K$-theory functors $K(-)$ and $\bbK(-)$.
\end{theorem}

\begin{proof}
The argument for Corollary~\ref{cor:dualmon} implies that, for any
infinite regular cardinal $\kappa$, the $\i$-category of functors
$\Fun((\idemstabcat)^\kappa, \ispec)$ is a symmetric monoidal
$\i$-category (again with respect to the convolution tensor product).
It is straightforward to check (again using the criterion
of~\cite[Proposition~2.2.1.6]{HA}) that this induces symmetric
monoidal structures on the subcategories
$\Fun_\add(\idemstabcat, \ispec)$ and
$\Fun_\loc(\idemstabcat, \ispec)$.  Specifically, the analogues of the
arguments for Proposition~\ref{prop:compat} and Proposition~\ref{JLoc}
hold with $((\idemstabcat)^{\kappa})^{\op}$ in place of
$(\idemstabcat)^\kappa$.
\end{proof}

Moreover, we have the following comparison result:

\begin{theorem}\label{thm:bodymainone}
The functor $\Umot$ induces a symmetric monoidal equivalence 
\[
\Fun^\L(\Motadd, \ispec)^{\otimes} \stackrel{\simeq}{\to} \Fun_\add(\idemstabcat,
\ispec)^{\otimes}.
\]
In particular, $\Umot$ induces an equivalence of
$E_\infty$ algebras in the two symmetric monoidal $\i$-categories.
Analogously, the functor $\Uloc$ induces a symmetric monoidal equivalence  
\[
\Fun^\L(\Motloc, \ispec)^{\otimes} \stackrel{\simeq}{\to} \Fun_\loc(\idemstabcat,
\ispec)^{\otimes}\,.
\]
\end{theorem}

\begin{proof}
Tracing through the constructions of the symmetric monoidal structures
and using the fact that $\Umot$ is symmetric monoidal, we see that it
induces a symmetric monoidal functor
\[
\Fun^\L(\Motadd, \ispec)^{\otimes} \to \Fun_\add(\idemstabcat,
\ispec)^{\otimes}. 
\]
Since we know it induces an equivalence on the underlying categories,
this implies it induces an equivalence of symmetric monoidal
$\i$-categories; see \cite[Remark~2.1.3.8]{HA}.  The argument for the
localizing case is analogous.
\end{proof}

Using recent work of Glasman~\cite{glasman}, we can rephrase the
preceding result to obtain the following relation between
$E_\infty$ algebras and lax symmetric monoidal functors.
Specifically, the main result of~\cite{glasman} establishes that there
is an equivalence 
\[
\Alg_{/
E_\infty}(\Fun(\aC, \ispec)^{\otimes}) \htp \Fun^{\lax}(\aC, \ispec),
\]
for any small $\i$-category $\aC$, where the functor category is given
the convolution symmetric monoidal structure.  Combining this
equivalence with the preceding result, we have the following:

\begin{corollary}\label{cor:laxcomp}
There are equivalences of $\i$-categories 
\[
\Alg_{/ E_{\infty}}(\Fun_{\add}(\idemstabcat, \aS_\infty)) \stackrel{\sim}{\too}
\Fun^{\lax}_{\mathrm{add}}(\idemstabcat,\aS_\infty)
\]
\[
\Alg_{/ E_{\infty}}(\Fun_{\loc}(\idemstabcat, \aS_\infty)) \stackrel{\sim}{\too}
\Fun^{\lax}_{\mathrm{loc}}(\idemstabcat,\aS_\infty)
\]
and
\[
\Alg_{/ E_\infty}(\Fun^{\L}(\Motadd, \aD)) \stackrel{\sim}{\too} \Fun^{\L,\lax}(\Motadd,\aD)) 
\]
\[
\Alg_{/ E_\infty}(\Fun^{\L}(\Motloc, \aD)) \stackrel{\sim}{\too} \Fun^{\L,\lax}(\Motloc,\aD)) 
\]
\end{corollary}

Corollary~\ref{cor:laxcomp} and Theorem~\ref{thm:bodymainone} leads to
the following comparison. 

\begin{theorem}
For any presentable symmetric monoidal $\i$-category $\aD$, there are
equivalences of $\i$-categories 
\[
(\Umot)^* \colon \Fun^{\L,\lax}(\Motadd,\aD) \stackrel{\sim}{\too}
\Fun^{\lax}_{\mathrm{add}}(\idemstabcat,\aD)
\]
\[
(\Uloc)^* \colon \Fun^{\L,\lax}(\Motloc,\aD) \stackrel{\sim}{\too}
\Fun^{\lax}_{\mathrm{loc}}(\idemstabcat,\aD)\,,
\]
where the left-hand sides denote the $\i$-category of lax symmetric
monoidal colimit-preserving functors and the right-hand sides denote
the $\i$-categories of lax symmetric monoidal additive or localizing
invariants, respectively.
\end{theorem}

\subsection*{The localizing subcategory generated by the
unit}\label{sec:localizing}

One interesting application of the symmetric monoidal structures on
$\Motadd$ and $\Motloc$ is the fact that these imply these categories
are enriched over the endomorphisms of the unit; i.e., algebraic
$K$-theory spectra of $\bbS$.  Specifically, the following result
follows from theorem~\ref{thm:sym} and from the
equivalences \eqref{eq:corep}.

\begin{corollary}\label{cor:Aenriched}
The symmetric monoidal homotopy categories $\mathrm{Ho}(\Motadd)$ and
$\mathrm{Ho}(\Motloc)$ are enriched over the homotopy category
$\mathrm{Ho}(A(*)\text{-}\Mod)$ of $A(*)$-modules.
\end{corollary}

In particular, if $\aA$ and $\aB$ are small stable $\i$-categories,
then the mapping spectrum $\Map(\Umot(\aB), \Umot(\aA))$ is a module
over  
\[
\Map(\Umot(\ispec^{\omega}), \Umot(\ispec^{\omega})) \htp K(\bbS) =
A(*).
\]
Similarly, the mapping spectrum $\Map(\Uloc(\aB), \Uloc(\aA))$ is a
module over
\[
\Map(\Uloc(\ispec^{\omega}), \Uloc(\ispec^{\omega})) \htp \bbK(\bbS) \htp A(*).
\]

\begin{remark}
In fact, it is possible to promote the enrichments of
$\mathrm{Ho}(\Motadd)$ and $\mathrm{Ho}(\Motloc)$ over
$\mathrm{Ho}(A(*)\text{-}\Mod)$ to enrichments of $\Motadd$ and
$\Motloc$ over the $\infty$-category $A(*)\text{-}\Mod$ of
$A(*)$-modules. This can be done directly, using the formalism of
$\infty$-operads as in~\cite[Definition~4.2.1.28]{HA}.  Briefly, the
action of $A(*)$ on the mapping spectra is given as follows.  The
endomorphism spectrum $\End({\bf 1})$ of the unit is an $A_\i$ ring
spectrum (even $E_\i$).  For objects $X$ and $Y$, the mapping spectrum
$F(X \otimes {\bf 1}, Y)$ is equivalent to $F(X,Y)$ and has an action
of $\End({\bf 1})$ via the composite map
\[
\End({\bf 1}) \to \End(X\otimes {\bf 1})\to F(X\otimes {\bf 1},Y).
\]
\end{remark}

In any symmetric monoidal stable $\i$-category $\aC$, we can consider
the smallest stable subcategory $\Loc_{\aC}({\bf 1})$ generated by the
unit object ${\bf 1}$ which is closed under (not necessarily finite)
direct sum; this is a lift of the localizing subcategory of the
homotopy category generated by the image of the unit.  If $\aC$ is
generated by the unit, then this subcategory is actually all of $\aC$;
in general, it is smaller.

Let $F_{\aC}(-,-)$ denote the mapping spectrum in $\aC$.  The 
endomorphism spectrum $\End_{\aC}({\bf 1}) = F_{\aC}({\bf 1},{\bf 1})$
is a commutative ring spectrum, and there is a functor 
\[
F_{\aC}({\bf 1},-)\colon \aC \too \End_{\aC}({\bf 1})\text{-}\mathrm{Mod}\,.
\]
When ${\bf 1}$ is a compact object in $\aC$, this functor induces an
equivalence between the category of modules over $\End_{\aC}({\bf 1})$
and $\Loc_{\aC}({\bf 1})$ (see \cite{DwyerGreenleesIyengar} for a nice
discussion of this kind of ``generalized Morita theory'').  Moreover,
there is an induced equivalence 
\[
F_{\aC}(X,Y) \stackrel{\sim}{\too} F_{\End_{\aC}({\bf
1})}(F_{\aC}({\bf 1},X),F_{\aC}({\bf 1},Y)).
\]
for every $X\in \Loc_{\aC}({\bf 1})$ and $Y \in \aC$.  Once consequence of
this is the following $\Ext$ spectral sequence (e.g., see \cite[4.1]{EKMM}):

\begin{corollary}
Given objects $X_1$ and $X_2$ in $\Loc_{\aC}({\bf 1})$, we have a
convergent spectral sequence 
\begin{align*}
E_2^{p,q}=\mathrm{Ext}^{p,q}_{\pi_{-*} \End_{\aC}({\bf 1})}\big(\pi_{-*} F_{\aC}({\bf
1},X_1), \pi_{-*} F_{\aC}({\bf 1},X_2) \big) \Rightarrow \\
\pi_{-p-q} F_{\End_{\aC}({\bf 1})}(F_{\aC}({\bf 1}, X_1), F_{\aC}({\bf 1}, X_2))
\end{align*}
and we can interpret both the $E_2$ term and the target in terms of
maps in $\aC$.
\end{corollary}

Since $\bbK(\bbS) \htp K(\bbS) = A(*)$, in our setting these localizing
subcategories can be identified with the $\i$-category of $A(*)$-modules.

\section{$THH$ and $TC$ as multiplicative theories}\label{sec:thhmult} 

In this section, we discuss a model for $THH$ which is a
multiplicative localizing invariant, i.e., an object of the
$\i$-category $\Alg_{/ E_\infty}(\Fun^\L(\Motloc, \ispec)^{\otimes})$.
We also explain how to interpret the construction of $TC$ in this
context; $TC$ itself is not an additive invariant as it does not
preserve filtered colimits, but constituents of $TC$ do yield additive
invariants.

More generally, we explain the passage from lax symmetric
monoidal functors from small spectral categories to spectra, such that the lax comparison maps are weak equivalences of spectra, to objects in
$\Alg_{/ E_\infty}(\Fun((\idemstabcat)^{\omega}, \ispec)^{\otimes})$;
see Theorem \ref{thm:laxdescends}.  If, in addition, the underlying functor is additive or localizing, then we shall see
Section~\ref{sec:unique} that these lax symmetric monoidal spectral-valued functors will admit unique multiplicative natural
transformations of additive functors from $K$ or localizing functors from $\bbK$.
We will be especially interested in the case of $THH$ and $TC$, which we will show admit concrete models as lax symmetric monoidal functors from small spectral categories to spectra.

We begin by producing a particular model of the convolution product on
the $\i$-category $\Fun((\idemstabcat)^{\omega}, \ispec)$.  Let
$\Spcat^{\cpt}$ denote a small full subcategory of $\Spcat$ which
contains a representative of each weak equivalence class of compact
objects, consists of pointwise-cofibrant objects, is closed under
smash product, and is closed under the functorial cofibrant-fibrant
replacement.  Let $\sPre((\Spcat^{\cpt})^{\op})$ denote the category
of presheaves of symmetric spectra (a.k.a. spectral presheaves) on
$(\Spcat^{\cpt})^{\op}$.  This has a symmetric monoidal product given
by the Day convolution.

\begin{proposition}\label{prop:mult}
With respect to the projective model structure and the convolution
symmetric monoidal structure, $\sPre((\Spcat^{\cpt})^{\op})$ is a
symmetric monoidal simplicial model category, as is the Bousfield
localization of $\sPre((\Spcat^{\cpt})^{\op})$ at the Morita
equivalences in $(\Spcat^{\cpt})^{\op}$, which we denote by
$\sPre^\mathrm{Mor}((\Spcat^{\cpt})^{\op})$.  Moreover, the symmetric
monoidal simplicial model structure on
$\sPre^\mathrm{Mor}((\Spcat^{\cpt})^{\op})$ induces an equivalence of
symmetric monoidal $\i$-categories
\[
(\N(\sPre^\mathrm{Mor}((\Spcat^{\cpt})^{\op}))[\aW^{-1}])^{\otimes} \htp
\Fun((\idemstabcat)^{\omega},\ispec)^{\otimes}.
\]  
\end{proposition}

\begin{proof}
The existence of the symmetric monoidal projective model structure on
$\sPre((\Spcat^{\cpt})^{\op}$ is straightforward (e.g.,
see~\cite[4.2]{DundasRoendigsOstvaer}).  The criterion
of \cite[5.8]{CisTab} and the fact that the (opposite of the) Morita
equivalences on $\Spcat^{\cpt}$ are closed under the smash product
imply that $\sPre^\mathrm{Mor}((\Spcat^{\cpt})^{\op})$ is a
symmetric monoidal model category.  The fact that the $\i$-categorical
Yoneda embedding is fully-faithful \cite[5.1.3.1]{HTT} implies that
there is an equivalence
\[
\N(\sPre^\mathrm{Mor}((\Spcat^{\cpt})^{\op}))[\aW^{-1}] \htp
\Fun((\idemstabcat)^{\omega},\ispec),
\]
and the universal property of the $\i$-categorical Day
convolution (see \cite[Corollary 6.3.1.12]{HA}) promotes the equivalence of
theorem~\ref{thm:symcomp} to an equivalence of symmetric monoidal
$\i$-categories
\[
(\N(\sPre^\mathrm{Mor}((\Spcat^{\cpt})^{\op}))[\aW^{-1}])^{\otimes} \htp
\Fun((\idemstabcat)^{\omega},\ispec)^{\otimes}.
\]  
\end{proof}

Although we do not know that the model structure on
$\sPre^\mathrm{Mor}((\Spcat^{\cpt})^{\op})$ satisfies the monoid
axiom, nonetheless the results of~\cite[\S 4]{Spitzweck} produce a
$J$-semi model category structure on the category of algebras over a
cofibrant $E_\infty$ operad $\aO$.  Such a structure suffices to
maintain homotopical control over the underlying $\infty$-category.
In particular, cofibrant-fibrant objects in the J-semi model category
of $\aO$-algebras forget to cofibrant-fibrant objects in the
underlying category.

Next, recall the following comparison \cite[22.1]{mmss}:

\begin{proposition}\label{mmss}
The category of commutative algebras in the presheaf category
$\sPre((\Spcat^{\cpt})^{\op})$ equipped with the (symmetric monoidal)
convolution product is equivalent to the category of lax symmetric
monoidal functors $\Spcat^{\cpt} \to \aS$.
\end{proposition}

We now use Proposition~\ref{cor:opeinf} and the J-semi model structure on
$E_\infty$ algebras in $\sPre^\mathrm{Mor}((\Spcat^{\cpt})^{\op})$ 
to deduce the following theorem.

\begin{theorem}\label{thm:laxdescends}
Let $E$ be a lax symmetric monoidal functor from spectral categories
to spectra.  Assume that $E$ preserves Morita equivalences of flat spectral
categories and that the induced functor
$\tilde{E} \colon \idemstabcat \to \ispec$ is an additive invariant.
Then $\tilde{E}$ naturally extends to an $E_\i$ algebra object of $\Fun_\add(\idemstabcat, \ispec)^{\otimes}$.  The analogous results for localizing invariants hold.
\end{theorem}

\begin{proof}
By proposition~\ref{mmss}, $E$, when restricted to $\Spcat^{\cpt}$,
yields a commutative algebra in $\sPre((\Spcat^{\cpt})^{\op})$.
Fixing a cofibrant-fibrant $E_\infty$-operad $\aO$, we can restrict
along the map from $\aO$ to the terminal operad to produce an
$\aO$-algebra structure on $E$.  After cofibrant-fibrant replacement
in the J-semi model structure on $\aO$-algebras in
$\sPre^\mathrm{Mor}((\Spcat^{\cpt})^{\op})$,
Proposition~\ref{cor:opeinf} now implies that $E$ descends to an
$E_\infty$ algebra in the symmetric monoidal $\i$-category
$\Fun((\idemstabcat)^{\omega},\ispec)$ such that the underlying
functor agrees with $\tilde{E}$ up to equivalence.

According to the results of
Section~\ref{sec:mon}, the additive localization is symmetric
monoidal, so we obtain a symmetric monoidal functor
\[
\Fun((\idemstabcat)^\omega,\ispec)^\otimes\too\Fun_\add((\idemstabcat)^\omega,\ispec)^\otimes\simeq\Fun^\L(\Motadd,\ispec)^\otimes. 
\]
Since $\tilde{E}$ is already an additive invariant, it comes from an
$E_\infty$ algebra in the symmetric monoidal $\i$-category
$\Fun_\add((\idemstabcat)^\omega,\ispec)^\otimes$ of additive functors
from compact idempotent-complete small stable $\i$-categories to
spectra.  But this latter symmetric monoidal $\i$-category is
(symmetric monoidal) equivalent to the symmetric monoidal
$\i$-category $\Fun^\L(\Motadd,\ispec)^\otimes$ by
proposition~\ref{thm:bodymainone}, so we may also regard $\tilde{E}$ as an
$E_\infty$ algebra here.
\end{proof}

\begin{remark}
Equivalently, Corollary~\ref{cor:laxcomp} implies that a functor
satisfying the conditions of the preceding theorem gives rise to a lax
symmetric monoidal additive functor of $\i$-categories.
\end{remark}

We now apply this work in the context of the topological Hochschild
homology of spectral categories, which can be constructed using a
version of the Hochschild-Mitchell cyclic nerve \cite[3.1]{BM}. 

\begin{definition}\label{def:cyclicnerve}
For a small spectral category $\aC$ let
\[
THH_{q}(\aC)=\bigvee \aC(c_{q-1},c_{q}) \sma \dotsb \sma
\aC(c_{0},c_{1}) \sma \aC(c_{q},c_{0}),
\]
where the sum is over the $(q+1)$-tuples $(c_{0},\dotsc,c_{q})$ of
objects of $\aC$.  This becomes a simplicial object using the usual
cyclic bar construction face and degeneracy maps.  We will write
$THH(\aC)$ for the geometric realization.
\end{definition}

We do not expect this construction to have the correct homotopy type
unless the spectral category $\aC$ is pointwise-cofibrant.

\begin{lemma}\label{lem:THHdesc}
The point-set functor $THH$ from Definition~\ref{def:cyclicnerve}
descends to a functor of $\i$-categories
\[
THH \colon \idemstabcat \to \ispec.
\] 
\end{lemma}

\begin{proof}
Restricting to pointwise-cofibrant spectral categories, the result
follows from the fact that $THH$ takes Morita equivalences of spectral
categories to weak equivalences~\cite[5.9, 5.12]{BM}. 
\end{proof}

\begin{remark}
We can construct $THH$ directly on the level of $\i$-categories as
follows.  (This perspective is closely related to the view taken
in \cite{BFN}, for instance).  A small stable $\i$-category $\aC$
determines an exact functor 
\[
\Map_\C \colon \C^{\op}\otimes\C\too\ispec
\]
given by the morphism spectra in $\C$ (cf. \cite[\S 3]{BGT}).
Extending by (filtered) colimits results in a colimit preserving
functor 
\[
\Theta_\C \colon \Fun^{\ex}(\C^{\op}\otimes\C,\ispec)\simeq\Fun^{\ex}(\C\otimes\C^{\op},\ispec)\simeq\Ind(\C^{\op}\otimes\C)\too\ispec
\]
(note that we must use the canonical equivalence
$\C^{\op}\otimes\C\simeq\C\otimes\C^{\op}$ in order to obtain the
first map in the above composite).  One can check that the value of
$\Theta_\C$ on $\Map_\C\in\Fun^{\ex}(\C^{\op}\otimes\C,\ispec)$ is
precisely the spectrum $THH(\C)$.  (See~\cite[5.2.3]{Toenswisk}
for a discussion of this in the context of DG categories.)
\end{remark}

For spectral categories $\aC_1, \aC_2, \ldots, \aC_n$, the standard
shuffle product maps induce $\Sigma_n$-equivariant equivalences  
\[
THH(\aC_1) \sma THH(\aC_2) \sma \cdots \sma THH(\aC_n) \to THH(\aC_1
\sma \aC_2 \sma \cdots \sma \aC_n).
\]
These maps are associative and unital, and therefore we deduce the
following lemma.

\begin{lemma}
The shuffle product maps make $THH$ a lax symmetric monoidal functor
from spectral categories to spectra.  
\end{lemma}

Applying theorem~\ref{thm:laxdescends} we obtain the following
corollary: 

\begin{corollary}\label{cor:thheinf}
The functor $THH$ yields an object of $\Alg_{/
E_\infty}(\Fun^\L(\Motloc, \ispec)^{\otimes})$ or equivalently an
object of $\Fun^{\lax}_{\mathrm{loc}}(\idemstabcat,\aS_\infty)$.
\end{corollary}

\begin{proof}
This follows from Lemma~\ref{lem:THHdesc}, the discussion above and
the fact that $THH$ descends to a localizing invariant: $THH$
preserves filtered homotopy colimits and sends exact sequences 
of small stable $\i$-categories to cofiber sequences of
spectra~\cite[7.1]{BM}.
\end{proof}

\begin{remark}
Since the shuffle product maps are equivalences,
Proposition~\ref{cor:indufunc} implies that $THH$ in fact
descends to a symmetric monoidal functor $\idemstabcat \to \ispec$.
\end{remark}

Establishing the analogue of Corollary~\ref{cor:thheinf} for
topological cyclic homology ($TC$) is somewhat more complicated.  For
one thing, $TC$ does not preserve filtered homotopy colimits of
spectral categories, and so cannot give rise to an additive or
localizing invariant.  Moreover, the model of $THH$ given in 
Definition~\ref{def:cyclicnerve} is not the construction that is used
to build $TC$, and so a different construction (and argument) is
needed to see that $TC$ is lax symmetric monoidal. 

We now review a different model of $THH$ of a spectral category which
is adapted to the construction of $TC$ and related invariants which do
preserve filtered homotopy colimits.  Our treatment is relatively
rapid; we refer the interested reader to~\cite[\S 5]{BM2} and \cite[\S
4]{BM} for a more detailed discussion.

The key observation that leads to the construction of $TC$ is the fact
that the cyclic bar construction of Definition~\ref{def:cyclicnerve}
is not just a simplicial set, but in fact a cyclic set; the cyclic
operator is given by ``rotating'' the smash factors.  As a
consequence, the geometric realization is a spectrum with an
$S^1$-action.  In fact, $THH$ can be constructed as an
$S^1$-equivariant spectrum equipped with additional structure that
models the structure of the free loop space.  Specifically, the $p$-th
root self-equivalences $S^1 / H \cong S^1$ (for finite $H \subset S^1$)
induces equivariant weak equivalences 
\[
\Map(S^1, X)^H \to \Map(S^1, X)
\]
for reasonable spaces $X$.  We are going to give a model of $THH$ as a
cyclotomic spectrum, which is a spectrum-level version of this
structure.  Roughly speaking, a cyclotomic spectrum is equipped with
compatible maps
\[
\Phi^H THH(\aC) \to THH(\aC)
\]
for finite $H \subset S^1$, where $\Phi^H$ denotes the geometric
fixed points of the $S^1$-spectrum.

We now fix a prime $p$ and consider the subgroups $H = C_{p^k}$ as $k$
varies.  The structure of a cyclotomic spectrum supplies the system of
``categorical'' fixed points $\{THH(\aC)^{C_{p^{n-1}}}\}$ with a pair
of maps 
\[
F,R\colon THH(\aC)^{C_{p^n}} \to THH(\aC)^{C_{p^{n-1}}},
\]
where $F$ (the Frobenius) denotes the obvious inclusion of fixed
points and $R$ (the restriction) is a much less obvious map coming
from the cyclotomic structure.

For convenience, we denote the fixed points to be
\[
TR^n(\aC) = THH(\aC)^{C_{p^{n-1}}},
\]
the categorical fixed points with respect to the induced $C_{p^{n-1}}$
action.  We then define $TC^{n}(\aC)$ to be the homotopy equalizer 
\[
\holim_{F,R} TR^n(\aC) \to TR^{n-1}(\aC).\]
and we finally define
\[
TC(\aC) = \holim_n TC^n (\aC),
\]
where we form the homotopy limit over the maps induced by the
restriction $R$; this definition is equivalent to the one originally
given in~\cite{BokstedtHsiangMadsen}.

The issue that arises with Definition~\ref{def:cyclicnerve} is that
traditionally it has been difficult to obtain the correct equivariant
homotopy type directly from the cyclic nerve and in particular build
the restriction maps.  Instead, we need to use a variant of
B{\"o}kstedt's original construction, which we now review.  Let $\aI$
denote the category of finite ordered sets, with objects $n
= \{1,\ldots,n\}$ and morphisms the injections.  For a symmetric
spectrum $X$, we will denote by $X_n$ the $n$th space.

\begin{definition}
Let $\aC$ be a small spectral category and $X$ a space.  Define the
functor $\aG(C,X)_{n_0, \ldots, n_q}$ from $\aI^{q+1}$ to spaces by
the formula 
\[
\aG(C,X)_{n_0, \ldots n_q} = \Omega^{n_0 + \ldots n_1 + \ldots +
n_q} \left( \bigvee \aC(c_{q-1}, c_q)_{n_q} \sma \ldots \sma \aC(c_0,
c_1)_{n_1} \sma \aC(c_q, c_0)_{n_0} \right)
\]
and define
\[
THH(\aC; X)_q = \hocolim_{\bar{n} \in \aI^{q+q}} \aG(C,X)_{\bar{n}}.
\]
\end{definition}

For fixed $X$, the spaces $THH(\aC; X)_q$ assemble into a cyclic space
with degeneracy map induced by the unit and face maps induced by the
composition in $\aC$.  This assignment is functorial in $X$, and so
restricting to spheres 
\[
S^n \cong S^1 \sma S^1 \sma \ldots \sma S^1
\]
defines a symmetric spectrum $THH(\aC) = \{THH(\aC;S^n)\}$.  When
$\aC$ is a point-wise cofibrant spectral category, an elaboration of
the work of Shipley~\cite{ShipleyTHH} shows that there is a natural
isomorphism in the stable category between this model of $THH$ and the
cyclic nerve of Definition~\ref{def:cyclicnerve}~\cite[3.5]{BM}.

Now we fix a complete $S^1$-universe $U$ (i.e., an infinite-dimensional
real $S^1$-inner product space that contains each irreducible
finite-dimensional representation infinitely many times).  Evaluating
at the representation spheres $S^V$, where $V$ is a
finite-dimensional real $S^1$-inner product space, the collection
$\{THH(\aC; S^V)\}$ forms an orthogonal $S^1$-spectrum.  Here each
space is given the diagonal $S^1$-action from the cyclic structure as
well as the action of $S^1$ on $S^V$.  

This realization of $THH(\aC)$ as an orthogonal $S^1$-spectrum is a
cyclotomic spectrum~\cite[\S 4]{BM}, and therefore can be used to
define $TC^n$ and $TC$.  It is possible to show that $THH(-)$ is a lax
symmetric monoidal functor regarded as a functor to cyclotomic spectra
in orthogonal spectra.  However, herein we employ a shortcut due to
Hesselholt-Madsen in order to obtain the multiplicative properties we
want.  Specifically, one can write down directly the restriction and
Frobenius maps as maps of non-equivariant symmetric spectra.

To capture the monoidal properties of these maps, we use a
generalization of B{\"o}kstedt's original construction of
$THH$~\cite[\S 1.7]{HesselholtMadsenwitt}.  Let $P$ be a finite
ordered set.  Define 
\[
\aG(C,X)_{n_0, \ldots, n_q}^P \colon (\aI^{P})^{q+1} \to \aT
\]
to be the functor determined by composing $\aG(C,X)$ with the functor
$\cup_{P} \colon \aI^{P} \to \aI$.  This construction is functorial in 
both $X$ with respect to continuous maps and $\aP$ with respect to
injective maps.  For fixed $P$, as above we can define an orthogonal
spectrum $THH^P(\aC)$.  B{\"o}kstedt's lemma about ``good indexing
categories'' implies that $THH^P(\aC)$ is equivalent to
$THH(\aC)$~\cite[1.6]{Bokstedt}.
Moreover, the proof of~\cite[1.7.1]{HesselholtMadsenwitt} extended to
spectral categories as in~\cite[\S 6.2]{GeisserHesselholt} gives rise
to an $S^1 \times \Sigma_m \times \Sigma_n$-equivariant multiplication
\begin{equation}\label{eq:boklax}
THH^P(\aC; X) \sma THH^Q(\aD; Y) \to THH^{P \coprod Q}(\aC \sma \aD;
X \sma Y), 
\end{equation}
where $|P| = m$ and $|Q| = n$.

Considering the collection of spaces $\{THH^P(\aC; S^n)\}$, where $|P|
= n$, we have a symmetric spectrum which is again equivalent to
$THH(\aC)$.  The multiplication map in equation~\eqref{eq:boklax} is
associative and unital, and so we have the following result:

\begin{lemma}
The construction $\{THH^P(\aC; S^n)\}$ gives rise to a lax symmetric
monoidal functor from spectral categories to symmetric spectra.  
\end{lemma}

Moreover, one can directly construct the restriction and Frobenius
maps $R_r, F_r$ on the fixed points of this model.  Furthermore, 
these maps are compatible with the product
structure~\cite[1.7.1]{HesselholtMadsenwitt}.  Therefore, we have the
following extension of~\cite[3.6]{HesselholtMadsenwitt}. 

\begin{theorem}
There are lax symmetric monoidal functors $TC^n$ and $TC$ from spectral
categories to symmetric spectra.
\end{theorem}

Although $TC$ does not preserve filtered colimits and therefore cannot
be a localizing invariant, using~\cite[10.8]{BGT} we do have the following
analogue of Corollary~\ref{cor:thheinf}.

\begin{corollary}\label{cor:tceinf}
For each $n$, the functor $TC^n$ from spectral categories to spectra
yields an object of $\Alg_{/
E_\infty}(\Fun^\L(\Motloc, \ispec)^{\otimes})$ or equivalently
$\Fun^{\lax}_{\mathrm{loc}}(\idemstabcat,\aS_\infty)$.
\end{corollary}

We conclude this section by remarking on what we did not prove.
Although we constructed $THH$ and $TC$ as lax symmetric monoidal
functors from spectral categories to spectra, we did not construct
point-set models of the topological Dennis trace or cyclotomic trace
that are compatible with the multiplicative structure.  To do so
involves ``mixing'' Waldhausen's $\Sdot$ construction with the
construction of $THH$ and $TC$, and handling multiplicative coherence
is quite intricate in this framework.  Although this can be carried
out using the coherence machinery of~\cite{BM3} (see
also~\cite[\S V.4]{DundasGoodwillieMcCarthy} for discussion of this
approach to multiplicative coherence), a distinct advantage of our
framework herein is that we will obtain existence and uniqueness
results for these multiplicative trace maps without an explicit
model.

\section{Uniqueness results}\label{sec:unique}

In this section, we apply the framework we have developed to deduce
various uniqueness results for multiplicative structures on algebraic
$K$-theory and apply the work of Section~\ref{sec:thhmult} to deduce
uniqueness and existence results for multiplicative natural
transformations out of algebraic $K$-theory.  In particular, our work
gives universal constructions of the topological Dennis trace $K \to
THH$ and cyclotomic trace $K \to TC$.

Using equivalences \eqref{eq:equivalenceadd}
and \eqref{eq:equivalenceloc}, Theorem~\ref{thm:maintwo} implies that
$\Fun^\L(\Motadd, \ispec)$ and $\Fun^\L(\Motloc, \ispec)$ are
symmetric monoidal $\i$-categories with tensor units the respective
algebraic $K$-theory functors. Hence, Corollary~\ref{cor:triv}
immediately implies the following result:

\begin{proposition}\label{prop:aux1}
\hspace{1 pt}
\begin{itemize} 
\item[(1)]  For any $n \geq 0$, the space $\Alg_{/ E_n}(\Map(\Umot(\ispec^\omega),-))$ of $E_n$-monoidal structures on the algebraic $K$-theory functor is
contractible.
\item[(2)] For any $n \geq 0$, the space of maps in $\Alg_{/ E_n}(\Fun^\L(\Motadd,
\ispec))$ with source $\Map(\Umot(\ispec^\omega),-)$ is
contractible.
\end{itemize}
Analogous results hold in the localizing case.

\end{proposition}

In particular, we obtain the following corollary:

\begin{corollary}\label{cor:bodymainthree}
There exists a unique $E_\infty$ algebra structure on the $K$-theory
functor, viewed as an object of the symmetric monoidal $\i$-category
$\Fun_\add(\idemstabcat, \ispec)^{\otimes}$.  Furthermore, for any
$0\leq n\leq\i$ and any $E_n$ algebra $F$, the space of $E_n$ algebra
maps from $K$ to $F$ is contractible.  Analogous statements hold for
$\bbK$.
\end{corollary}

Coupled with Theorem~\ref{thm:new}, this specializes into our main
application:

\begin{theorem}\label{thm:cormainone}
The space of maps of $E_\infty$ algebras in
$\Fun_\add(\idemstabcat, \ispec)^{\otimes}$ from $K$-theory to $THH$
is contractible.  Equivalently, the space of lax symmetric monoidal
additive functors from $K$-theory to $THH$ is contractible.  The
unique element is the topological Dennis trace map.
\end{theorem}

We identify the unique element as the topological Dennis trace map
using the main classification result of~\cite[10.6]{BGT}.  Specifically, the
image of this element under the forgetful functor to natural
transformations of additive functors is the unit in the set $\bZ$ of
homotopy classes of natural transformations of additive functors
$K \to THH$.  Note that our results in fact give a construction of the
multiplicative trace in the generality of stable $\i$-categories.

We can deduce analogous results about $TC$ even though $TC$ is not
itself a localizing invariant.  As in~\cite[10.11]{BGT} we use the
fact that $TC$ can be described as $\holim_n TC^n(-)$.  As such, any
natural transformation of functors $K \to TC$ is equivalent to the
data of compatible maps to each $TC^n$.  Corollary~\ref{cor:tceinf}
now implies that our uniqueness results extend to natural
transformations which arise from transformations $K \to TC^n$.  Since
the cyclotomic trace is of this form, we deduce the desired uniqueness
result.

\begin{theorem}\label{thm:cormaintwo}
For each $n$, the space of $E_\infty$ algebra maps from $K$ to $TC^n$ in
$\Fun_\add(\idemstabcat, \ispec)^{\otimes}$ is contractible.
Equivalently, the space of lax symmetric monoidal additive functors
from $K$ to $TC^n$ is contractible.  The unique homotopy class of maps
of $E_\infty$ algebras in $\Fun(\idemstabcat, \ispec)^{\otimes}$ from
$K$ to $TC$ that restrict to maps of $E_\infty$ algebras $K \to TC^n$
is the multiplicative cyclotomic trace.
\end{theorem}

Another application of Corollary~\ref{cor:bodymainthree} is the
following uniqueness result:

\begin{corollary}\label{cor:bodymaintwo}
Let $\aC$ be a symmetric monoidal spectral category, and
$\Perf(\aC)$ be the resulting symmetric monoidal category of
compact modules.  There is a unique $E_\infty$ algebra structure on
$K(\aC)$ in the $\i$-category of spectra.  Similarly, if $\aC$ is a
monoidal spectral category, there is a unique $A_\infty$ structure on 
$K(\aC)$ in the $\i$-category of spectra.  Analogous results hold for
$\bbK$.
\end{corollary}

\begin{proof}
Since $K$-theory has a unique structure as a lax symmetric monoidal
functor from $\idemstabcat$ to $\ispec$, in particular when evaluated
at any point there is a unique $E_\infty$ or $A_\infty$ structure on
the resulting spectrum.
\end{proof}

\end{document}